\documentclass[12pt]{amsart}
\usepackage{a4wide,enumerate,xcolor}
\usepackage{amsmath,comment,graphicx}
\allowdisplaybreaks
\usepackage{enumitem}
\setlist[enumerate]{label={\rm(\arabic*)}, leftmargin=28pt, itemsep=3pt}
\setlist[itemize]{label={$\bullet$}, leftmargin=12pt, itemsep=3pt}

\let\pa\partial
\let\na\nabla
\let\eps\varepsilon

\newcommand{\R}{{\mathbb R}}
\newcommand{\diver}{\operatorname{div}}

\newcommand{\dom}{\mathcal{D}}

\newtheorem{theorem}{Theorem}
\newtheorem{lemma}[theorem]{Lemma}
\newtheorem{proposition}[theorem]{Proposition}
\newtheorem{remark}[theorem]{Remark}

\newtheorem{example}{Example}

\begin{document}

\title[Multiphase cross-diffusion models]{Multiphase cross-diffusion 
models for \\ tissue structures: modeling, analysis, numerics}

\author[A. J\"ungel]{Ansgar J\"ungel}
\address{Institute of Analysis and Scientific Computing, TU Wien, Wiedner Hauptstra\ss e 8--10, 1040 Wien, Austria}
\email{juengel@tuwien.ac.at} 

\author[C. Reisch]{Cordula Reisch}
\address{Department of Science and Environment, Roskilde University, Universitetsvej 1, 4000 Roskilde, Denmark}
\email{cordular@ruc.dk}

\author[S. Xhahysa]{Sara Xhahysa}
\address{Institute of Analysis and Scientific Computing, TU Wien, Wiedner Hauptstra\ss e 8--10, 1040 Wien, Austria}
\email{sara.xhahysa@tuwien.ac.at} 

\date{\today}

\thanks{The first and last authors acknowledge partial support from   
the Austrian Science Fund (FWF), grant 10.55776/F65 and 10.55776/PAT2687825, and from the Austrian Federal Ministry for Women, Science and Research and implemented by \"OAD, project MultHeFlo. This work has received funding from the European Research Council (ERC) under the European Union's Horizon 2020 research and innovation programme, ERC Advanced Grant NEUROMORPH, no.~101018153. For open access purposes, the authors have applied a CC BY public copyright license to any author-accepted manuscript version arising from this submission.} 

\begin{abstract}
Volume-filling cross-diffusion equations for the components of a tissue structure are formally derived from mass conservation laws and force balances for the interphase pressures and viscous drag forces in a multiphase approach. The equations include Maxwell--Stefan, tumor-growth, thin-film solar cell models as well as novel volume-filling population systems. The Boltzmann and Rao entropy structures are explored. If the drag coefficients are all equal to one, the global-in-time existence of bounded weak solutions, their long-time behavior, and the weak--strong uniqueness of solutions to a regularized system are proved using entropy methods. In the general case, the resulting diffusion matrix is positively stable, ensuring local-in-time existence of solutions. Global-in-time existence of weak solutions is proved if the drag coefficients are sufficiently close to each other. This restriction is explained by the fact that the pressure forces are of degenerate type, while the drag forces are nondegenerate in the volume fractions. Numerical simulations are presented in one space dimension to illustrate the solution behavior beyond the entropy regime.
\end{abstract}

 

\maketitle


\section{Introduction}

Multiphase models provide a continuum framework to investigate interactions between different components of a cellular fluid. Multiphase characterizes situations where several different phases (e.g.\ cells, water, extracellular matrix) are flowing simultaneously. The phases are present at every material point and momentum and mass balance equations are postulated. Assuming force balances for the interphase pressures and viscous drag forces like in \cite{LKBJS06}, the evolution of the volume fractions of the phases is governed by cross-diffusion equations. In this paper, we derive formally a novel class of cross-diffusion systems, analyze the cases when the drag coefficients are all the same or not, and present some numerical experiments. 

\subsection{Model setting}

The equations for the volume fraction $u_i(x,t)\in[0,1]$ and velocity $v_i(x,t)\in\R^d$ of the $i$th phase are given by 
\begin{align}
  & \pa_t u_i + \diver(u_iv_i) = 0\quad\mbox{in }\Omega,\ 
  t>0,\ i=0,\ldots,n, \label{1.mass} \\
  & \sum_{j=0}^n k_{ij}u_iu_j(v_i-v_j) 
  = -\bigg(\na(u_iq_i(u)) - u_i\sum_{j=0}^n\na(u_jq_j(u))\bigg) \label{1.mom} \\
  &\phantom{\sum}+ \bigg((r_i(u)\na u_i-u_i\na r_i(u))
  - u_i\sum_{j=0}^n\big(r_j(u)\na u_j-u_j\na r_j(u)\big)\bigg), 
  \nonumber 
\end{align}
where $u=(u_1,\ldots,u_n)$ are the phase components, $u_0$ is the solvent (typically, water), and
\begin{align}\label{1.qr}
  q_i(u) = \sum_{j=0}^n q_{ij}u_j, \quad
  r_i(u) = \sum_{j=0}^n r_{ij}u_j
\end{align}
are the intraphase and interphase pressures, respectively; see Section \ref{sec.model} for details. The equations are solved in a bounded domain $\Omega\subset\R^d$ ($d\ge 1$). The drag coefficients $k_{ij}>0$ and pressure coefficients $r_{ij}\ge 0$ are symmetric; the pressure coefficients $q_{ij}\ge 0$ may be functions of $u$, and they may be not symmetric. Reaction rates can be added as source terms to the right-hand side of \eqref{1.mass}, but we neglect them to simplify the presentation. We impose the initial and no-flux boundary conditions
\begin{align}\label{1.bic}
  u_i(0)=u_i^0\quad\mbox{in }\Omega, \quad
  u_iv_i\cdot\nu = 0 \quad\mbox{on }\pa\Omega,\ t>0,\ i=0,\ldots,n.
\end{align}
Typically, $u_1,\ldots,u_n$ are the volume fractions of cell phases (like tissue cells, tumor cells, necrotic cells, extracellular matrix, etc.) and $u_0$ is the liquid phase (water). 

To ensure total mass conservation, we assume that the total flux vanishes, $\sum_{i=0}^n u_iv_i=0$. Then, supposing that the initial total mass equals $\sum_{i=0}^n u_i(0)=1$, equations \eqref{1.mass} imply that $\sum_{i=0}^n u_i(t)=1$ for all $t>0$. Biologically, this means that the mixture is saturated. We call this property volume filling. As a consequence, we can replace the liquid phase by $u_0=1-\sum_{i=1}^n u_i$, and equations \eqref{1.mass}--\eqref{1.mom} can be formulated in the variable $u=(u_1,\ldots,u_n)$.  

Equations \eqref{1.mass}--\eqref{1.mom} form a cross-diffusion system. To see this, let first $k_{ij}=1$ for all $i,j=0,\ldots,n$, $i\neq j$. Then the left-hand side of \eqref{1.mom} equals $u_iv_i$ and inserting the right-hand side into \eqref{1.mass} yields
\begin{align}\label{1.samek}
  \pa_t u_i &= \diver\bigg(\na(u_iq_i(u)) 
  - u_i\sum_{j=0}^n\na(u_jq_j(u))\bigg) \\
  &\phantom{xx}- \diver\bigg(r_i(u)\na u_i-u_i\na r_i(u)
  - u_i\sum_{j=0}^n\big(r_j(u)\na u_j-u_j\na r_j(u)\big)\bigg),
  \nonumber 
\end{align}
where $i=0,\ldots,n$. After replacing $u_0$, these equations can be written compactly as the cross-diffusion system
\begin{align*}
  \pa_t u = \diver(A(u)\na u)\quad\mbox{in }\Omega,\ t>0,
\end{align*}
where $A(u)$ is an $n\times n$ matrix (whose entries are defined in \eqref{4.A} below). If $k_{ij}\neq 1$ are symmetric but not all equal, the left-hand side is a linear system in the fluxes $u_iv_i$ for $i=0,\ldots,n$. As the sum of the left-hand side of \eqref{1.mom} over $i=0,\ldots,n$ vanishes (since $k_{ij}$ is symmetric), the linear system cannot be solved uniquely. This issue can be overcome by either using the Bott--Duffin inverse (as in \cite{HJT22}) or by removing the equation for $u_0$ and solving the linear system uniquely in $u_iv_i$ for $i=1,\ldots,n$ (as in \cite{JuSt13}). In the latter case, we can write the left-hand side of \eqref{1.mom} as $(K(u)J)_i$, where $K(u)\in\R^{n\times n}$ is defined in \eqref{5.K} below and $J=(J_1,\ldots,J_n)$ with $J_i=u_iv_i$ being the flux of the $i$th phase. The matrix $K(u)$ is invertible, and we can write \eqref{1.mass}--\eqref{1.mom} compactly as the cross-diffusion system
\begin{align}\label{1.K1A} 
  \pa_t u = \diver(K(u)^{-1}A(u)\na u)\quad\mbox{in }\Omega,\ t>0.
\end{align}
Notice that the diffusion matrices $A(u)$ and $K(u)^{-1}A(u)$ are generally neither symmetric nor positive (semi-) definite such that classical PDE tools cannot be applied. We overcome this issue by revealing and exploiting the entropy structure associated to \eqref{1.mass}--\eqref{1.mom}; see Section \ref{sec.ent}.

Equations \eqref{1.mass}--\eqref{1.mom} include several known and new models:
\begin{itemize}
\item Maxwell--Stefan equations: If $q_i(u)=1$ and $r_i(u)=0$, we recover the equations
\begin{align*}
  \pa_t u_i + \diver(u_iv_i) = 0, \quad
  -\na u_i = \sum_{j=0}^n k_{ij}u_iu_j(v_i-v_j), \quad i=0,\ldots,n.
\end{align*}
They describe the dynamics of gaseous mixtures and arise in many fields like dialysis, respiratory airways, electrolysis, and chemical reactors \cite{WeKr00}. The local (in time) existence of solutions to the Maxwell--Stefan equations was proved in \cite{Bot11,GoMa98}, the global (in time) existence of weak solutions in \cite{JuSt13}. Notice that the differential operator of the Maxwell--Stefan equations is nondegenerate, i.e., it contains Laplacian terms. 
\item Thin-film solar cell equations: We choose $k_{ij}=1$ and $q_i(u)=0$. Then equations \eqref{1.samek} become
\begin{align*}
  \pa_t u_i = \diver\bigg(\sum_{j=0}^n r_{ij}
  (u_j\na u_i-u_i\na u_j)\bigg), \quad i=0,\ldots,n.
\end{align*}
They arise in the production of solar cells by physical vapor deposition. In this process, some material components with volume fraction $u_i$ are evaporated in a high-temperature vacuum chamber. The incoming atoms deposit on the substrate and form a thin-film layer, which is used as the absorption layer of the solar cell. The equations have been suggested and analyzed in \cite{BaEh18}. 
\item Volume-filling population equations: Setting $k_{ij}=1$ and $r_i(u)=0$, we obtain the equations
\begin{align}\label{1.vfSKT}
  \pa_t u_i = \diver\bigg(\na(u_iq_i(u)) 
  - u_i\sum_{j=0}^n\na(u_jq_j(u))\bigg), \quad i=0,\ldots,n.
\end{align}
Scaling $\tau=t/\delta$, $w_i=\delta u_i$ for $\delta>0$ and performing the formal limit $\delta\to 0$ leads to the Shigesada--Kawasaki--Teramoto (SKT) population model $\pa_{\tau}w_i=\Delta(w_iq_i(w))$ \cite{SKT79}. In this sense, equations \eqref{1.vfSKT} are a version of the SKT model yielding bounded solutions (since $\sum_{i=0}^n u_i=1$). The equations contain quadratic degenerate expressions of the type $\Delta(u_i u_j)$. 
Note that the boundedness of solutions to the SKT system can be proved only for special values of $q_{ij}$ \cite{JuZa16}, which justifies the study of the variant \eqref{1.vfSKT} of the SKT model. This variant is new.
\item Volume-filling Busenberg--Travis equations: We choose $k_{ij}=1$ and $q_{ij}=r_{ij}$, leading to $q_i(u)=r_i(u)$ and hence
\begin{align}\label{1.vfBT}
  \pa_t u_i = 2\diver\bigg(u_i\na q_i(u) 
  - u_i\sum_{j=0}^n u_j\na q_j(u)\bigg), \quad i=0,\ldots,n.
\end{align}
With the same scaling and asymptotic limit as in the previous example, we recover the generalized Busenberg--Travis equations $\pa_\tau w_i=\diver(w_i\na q_i(w))$ \cite{BuTr83}. Bounded weak solutions to this model are known to exist for $n=2$ only \cite{LaMa23}. Equations \eqref{1.vfBT} yield bounded weak solutions for any $n\ge 2$; the model is new.
\item Tumor-growth cross-diffusion equations: We choose $n=2$, $k_{ij}=1$, $r_{ij}=0$, and the intraphase pressures
\begin{align*}
  q_{11} = \beta, \quad q_{12} = 0, \quad q_{21} = \beta\theta u_2, 
  \quad q_{22} = \beta,
\end{align*}
where $\beta>0$, $\theta>0$ are some constant parameters. The equations become
\begin{equation}\label{1.tg}
\begin{aligned}
  \pa_t u_1 &= 2\beta\diver\big[\big(u_1(1-u_1)-\theta u_1u_2^2\big)
  \na u_1 - u_1u_2(1+\theta u_1)\na u_2\big], \\
  \pa_t u_2 &= 2\beta\diver\big[\big(-u_1u_2+\theta(1-u_2)u_2^2\big)
  \na u_1 + u_2(1-u_2)(1+\theta u_1)\na u_2\big].
\end{aligned}
\end{equation}
This model was formally derived in \cite{JaBy02} and analyzed in \cite{JuSt12}. In this example, the pressure coefficient $q_{21}$ depends on the volume fraction $u_2$, which models the extracellular matrix.
\end{itemize}

In the following, we present our main results. We consider equal drag coefficients $k_{ij}=1$ and unequal coefficients $k_{ij}$ separately. 


\subsection{Equal drag coefficients}

We prove the existence of global solutions, their long-time behavior, and, under additional conditions, a weak--strong uniqueness property. We impose the following general assumptions:
\begin{enumerate}
\item[(A1)] Domain: $\Omega\subset\R^d$ ($d\ge 1$) is a bounded domain with Lipschitz boundary and let $T>0$. Set $\Omega_T=\Omega\times(0,T)$ and $\dom:=\{u\in(0,1)^n:\sum_{i=1}^n u_i<1\}$.
\item[(A2)] Initial data: $u^0=(u_1^0,\ldots,u_n^0)\in L^1(\Omega;\R^n)$ satisfies $u^0(x)\in\overline{\dom}$ for a.e.\ $x\in\Omega$. 
\item[(A3)] Drag coefficients: $k_{ij}=1$ for $i\neq j$ and $k_{ii}=0$ for $i=0,\ldots,n$. 
\item[(A4)] Pressure coefficients: $q_{ij}\ge 0$, $r_{ij}=r_{ji}\ge 0$,  $q_{i0}=q_{0i}=r_{i0}=r_{ii}=0$ for $i,j=0,\ldots,n$.
\end{enumerate}

As discussed earlier, the symmetry of $r_{ij}$ (and $k_{ij}$) is required to ensure that the total force vanishes. We set $k_{ii}=r_{ii}=0$ to allow for sums going from $i=0,\ldots,n$. The assumption $q_{i0}=r_{i0}=0$ is made for convenience; it simplifies the reduction from $n+1$ to $n$ species. Biologically, this means that the water phase does not contribute to the pressure. 

In the following, we call a possibly nonsymmetric matrix $A$ positive definite if its symmetric part $\frac12(A^T+A)$ is positive definite. We start with the global existence result.

\begin{theorem}[Global existence of solutions]\label{thm.ex}
Let Assumptions (A1)--(A4) hold, the matrix $(q_{ij}+r_{ij})_{i,j=1}^n$ be positive definite with smallest eigenvalue $\alpha>0$, and $q_{ij}\ge r_{ij}$ for $i,j=1,\ldots,n$. Then there exists a global bounded weak solution $u$ to \eqref{1.qr}--\eqref{1.samek} satisfying $u(x,t)\in\overline{\dom}$ for a.e.\ $(x,t)\in\Omega_T$ and
\begin{align*}
  u\in L^\infty(\Omega_T;\R^n)\cap L^2(0,T;H^1(\Omega;\R^n)), \quad
  \pa_t u\in L^2(0,T;H^1(\Omega;\R^n)'). 
\end{align*}
Moreover, the entropy inequality holds for $0<t<T$:
\begin{align}\label{1.ent}
  \int_\Omega h_B(u(t))dx + \alpha\sum_{i=0}^n\int_0^t\int_\Omega
  |\na u_i|^2 dxds \le \int_\Omega h_B(u^0)dx,
\end{align}
where $h_B(u)=\sum_{i=0}^n u_i(\log u_i-1)$ and $u_0=1-\sum_{i=1}^n u_i$.
\end{theorem}

The theorem is proved by applying the boundedness-by-entropy method \cite{Jue15}. The idea is to derive a priori estimates from the Boltzmann entropy density $h_B(u)$, leading to the entropy inequality \eqref{1.ent}. This is possible since $h_B''(u)A(u)$ is positive definite uniformly in $\dom$. The pointwise bound $u(x,t)\in\overline{\dom}$ is also a consequence of the Boltzmann entropy. Indeed, the idea is to solve an approximation of the cross-diffusion equations in terms of the variable $w=(w_1,\ldots,w_n)=h_B'(u)$. Then the volume fraction $u_i$ is a function of $w$ via
\begin{align*}
  u_i = \frac{\exp w_i}{1 + \sum_{j=1}^n \exp w_j}, \quad i=1,\ldots,n,
\end{align*}
which shows that $u\in\dom$; see Section \ref{sec.ex}. We have formulated Theorem \ref{thm.ex} for constant coefficients $q_{ij}$ and $r_{ij}$ for simplicity, but in fact, they are allowed to depend on $u$; see Remark \ref{rem.qr}. The main hypothesis is that $(q_{ij}+r_{ij})_{i,j=1}^n$ be positive definite uniformly in $\dom$. This is important for instance in the tumor-growth model \eqref{1.tg}, where the coefficient $q_{21}$ depends on $u_2$; see Section \ref{sec.tg}. 

If $(q_{ij}+r_{ij})_{i,j=1}^n$ is not positive definite, we cannot expect even local existence of solutions. Indeed, consider the tumor-growth model \eqref{1.tg}. We show in Section \ref{sec.tg} that for sufficiently large $\theta>0$, the diffusion matrix $A(u)$ possesses negative eigenvalues for certain values of $u\in\dom$, meaning that the system is generally {\em not} parabolic in the sense of Petrovskii, which is considered to be a minimal condition for (local) solvability \cite{Ama90}. Notice that if $\theta<4/\sqrt{\beta}$ then the existence of global weak solutions can be shown by the entropy method \cite{JuSt12}.

Next, we show that the solution $u_i(t)$ converges exponentially fast to the constant steady state $u_i^\infty = |\Omega|^{-1}\int_\Omega u_i^0 dx$ for $i=0,\ldots,n$. Set $u^\infty=(u_1^\infty,\ldots,u_n^\infty)$. 

\begin{theorem}[Exponential decay]\label{thm.decay}
Let the assumptions of Theorem \ref{thm.ex} hold, let $u$ be a bound\-ed weak solution to \eqref{1.qr}--\eqref{1.samek}, and assume that $u^\infty\in\dom$. Then there exists $\lambda>0$, depending on $u^\infty$, $d$, $\Omega$, and the smallest eigenvalue of $(q_{ij}+r_{ij})_{i,j=1}^n$ such that
\begin{align}\label{1.decay}
  \|u(t)-u^\infty\|_{L^1(\Omega)} 
  \le \sqrt{2\|u^0\|_{L^1(\Omega)}H_B(u^0|u^\infty)}
  e^{-\lambda t} \quad\mbox{for }t>0,
\end{align}
where
\begin{align}\label{1.HB}
  H_B(u|u^\infty) = \sum_{i=0}^n\int_\Omega 
  u_i\log\frac{u_i}{u_i^\infty}dx
\end{align}
is the relative Boltzmann entropy.
\end{theorem}

The idea of the proof is to estimate the time derivative of the relative entropy:
\begin{align*}
  \frac{d}{dt}H_B(u|u^\infty) \le -\alpha\sum_{i=1}^n
  \int_\Omega|\na u_i|^2 dx \le -\alpha C(u^\infty)H_B(u|u^\infty),
\end{align*}
where the first inequality is a consequence of \eqref{1.ent} and the 
second inequality follows from a nonhomogeneous variant of the logarithmic Sobolev inequality. Then Gronwall's lemma and the Csisz\'ar--Kullback--Pinsker inequality lead to \eqref{1.decay}; see Section \ref{sec.decay} for details. 

The uniqueness of weak solutions to cross-diffusion systems is a very difficult problem and has been proven so far only for very special systems; see, e.g., \cite{ChJu18}. An alternative is the weak--strong uniqueness property. The aim is to show that a weak solution and a strong solution (if it exists) emanating from the same initial data coincide. Weak--strong uniqueness means that classical solutions are stable within the class of weak solutions. Usually, the weak--strong uniqueness is proved by means of the relative Boltzmann entropy; see, e.g., \cite{HeJu25}. However, it turns out that the relative Boltzmann entropy cannot be used here. This issue is resolved by employing another entropy and a regularized version of the equations. More precisely, the choice $q_i=r_i$, leading to \eqref{1.vfBT}, ensures that the so-called Rao entropy (see below) allows us to apply the relative entropy method. Still, certain gradient terms need to be estimated, which is only possible if we add a diffusive regularization to \eqref{1.vfBT}. Such a regularization is often used in numerics to stabilize the discretization.

\begin{theorem}[Weak--strong uniqueness]\label{thm.wsu}
Let the assumptions of Theorem \ref{thm.ex} hold and assume that $q_{ij}=r_{ij}$ in \eqref{1.qr} for $i,j=1,\ldots,n$. Let $u$ be a bounded weak solution to
\begin{align}\label{1.regBT}
  \pa_t u_i = \eta\Delta u_i + 2\diver\bigg(u_i\na q_i(u)
  - u_i\sum_{j=0}^n u_j\na q_j(u)\bigg), \quad i=0,\ldots,n,
\end{align}
with the initial and no-flux boundary conditions \eqref{1.bic}, where $\eta>0$ is arbitrary. Furthermore, let $\bar{u}$ be a nonnegative strong solution to this problem with the same initial data and the regularity $\na\bar{u}\in L^\infty(\Omega_T)$. Then $u(t)=\bar{u}(t)$ in $\Omega$ for $t\ge 0$. 
\end{theorem}

This theorem is proved by estimating the relative Rao entropy 
\begin{align*}
  H_R(u|\bar{u}) = \sum_{i,j=1}^n\int_\Omega q_{ij}(u_i-\bar{u}_i)
  (u_j-\bar{u}_j)dx
\end{align*}
between a weak solution $u$ and a strong solution $\bar{u}$. A computation shows for any $\delta>0$ that (see Section \ref{sec.wsu})
\begin{align*}
  \frac{d}{dt}H_R(u|\bar{u}) 
  &\le -\eta\sum_{i=1}^n\int_\Omega|\na(u_i-\bar{u}_i)|^2 dx
  + \delta\sum_{i=1}^n\int_\Omega|\na(q_i(u)-q_i(\bar{u}))|^2 dx \\
  &\phantom{xx}+ C(\delta)\sum_{i=1}^n\int_\Omega|u_i-\bar{u}_i|^2 dx.
\end{align*}
The second term on the right-hand side can be absorbed by the first term if $\delta>0$ is sufficiently small. At this point, we need the assumption $\eta>0$, which is not needed for the existence result. The last term on the right-hand side is bounded by the relative Rao entropy (since $(q_{ij})_{i,j=1}^n$ is assumed to be positive definite). This leads to
\begin{align*}
  \frac{d}{dt}H_R(u|\bar{u}) \le CH_R(u|\bar{u}), \quad t>0,
\end{align*}
and since $H_R(u(0)|\bar{u}(0))=0$ by assumption, we conclude from Gronwall's lemma that $H_R(u(t)|\bar{u}(t))=0$ and hence $u(t)=\bar{u}(t)$ in $\Omega$ for $t>0$.


\subsection{Varying drag coefficients}

The case of different drag coefficients $k_{ij}$ is delicate, since the entropy method does not apply in the general case. We consider the cross-diffusion system \eqref{1.K1A}, where $K(u)$ is no longer the unit matrix. We are able to show the following results:

\begin{itemize}
\item The diffusion matrix $K(u)^{-1}A(u)$ in \eqref{1.K1A} is positively stable, i.e., the real parts of its eigenvalues are positive (Proposition \ref{prop.posstab}). This means that \eqref{1.K1A} is parabolic in the sense of Petrovskii. By Amann's theory \cite{Ama90}, for smooth initial data, system \eqref{1.mass}--\eqref{1.bic} possesses a unique local classical solution.
\item If the values $k_{ij}$ are a perturbation of the number one, the matrix $h''_B(u) K(u)^{-1}A(u)$ is positive definite for $u\in\dom$ (Proposition \ref{prop.perturb}). By the boundedness-by-entropy method, there exists a global bounded weak solution to \eqref{1.mass}--\eqref{1.bic}.
\item In case $n=2$ and $r_1=r_2=0$, we determine conditions on $(k_{ij})$ and $(q_{ij})$ such that the matrix $h''_B(u)K(u)^{-1}A(u)$ is positive definite. Then the boundedness-by-entropy method can be applied, yielding the existence of a global bounded weak solution. However, for general $(k_{ij})$ and $(q_{ij})$, the matrix $h''_B(u)K(u)^{-1}A(u)$ may have a negative eigenvalue, i.e., it is generally not positive definite. 
\end{itemize}

Our results show that equations \eqref{1.mass}--\eqref{1.bic} always have a local solution and, if $k_{ij}$ is close to one, there exists a global solution. However, in the general case, the entropy method cannot be easily applied. A perturbation result was also shown in \cite[Theorem 4.3]{KSZ21} assuming $k_{ij}=k_ik_j$ for some $k_i>0$. Our result holds in a more general setting. 

The results may be explained by the fact that the pressure forces are of quadratic degenerate type, like $u_i\na u_j$, while the drag forces are of nondegenerate type, which yields some kind of incompatibility close to vacuum $u_i=0$ in the estimations; see Remark \ref{rem.comb}. 

The paper is organized as follows. We derive equations \eqref{1.mass}--\eqref{1.mom} formally from multiphase compressible Navier--Stokes equations and explore the Boltzmann and Rao entropy structure in Section \ref{sec.model}. Our main results in case $k_{ij}=1$ (Theorems \ref{thm.ex}--\ref{thm.wsu}) are proved in Section \ref{sec.proofs}, while in Section \ref{sec.unsamek}, our results for varying coefficients $k_{ij}\neq 1$ are shown. We present some numerical experiments for the one-dimensional equations in Section \ref{sec.num}. Appendix \ref{app} is devoted to the technical proof of Proposition \ref{prop.G}. 


\section{Multiphase modeling}\label{sec.model}

We follow the multiphase approach of \cite{LKBJS06}. Our starting point are the multiphase compressible Navier--Stokes equations for the volume fraction $u_i$ of the $i$th phase and its velocity $v_i$,
\begin{align}
  & \pa_t u_i + \diver(u_iv_i) = 0, \quad i=0,\ldots,n, 
  \nonumber \\
  & \eps\pa_t (u_iv_i) + \eps\diver(u_iv_i\otimes v_i)
  - \diver(u_i\mathbb{S}_i) = F_i. \label{2.mom}
\end{align}
Here, $F_i$ represents the forces acting between pairs of phases, $u_i\mathbb{S}_i$ is the stress tensor of the $i$th phase, and $\eps>0$ represents a small-velocity scaling. As cells are mainly composed of water, their densities are assumed to be constant and having the same value \cite{LOK22}. Thus, the densities can be removed from the equations, only keeping the phase fraction $u_i$. We suppose that the total flow vanishes, $\sum_{i=0}^n u_iv_i=0$. As argued in the introduction, the property $\sum_{i=0}^n u_i(0)=1$ implies the volume-filling constraint $\sum_{i=0}^n u_i(t)=1$ for all $t>0$. Thus, the vector $u=(u_1,\ldots,u_n)$ contains the full information since $u_0=1-\sum_{i=1}^n u_i$.

We still need to determine the stress tensor $u_i\mathbb{S}_i$ and the interaction force $F_i$. We assume that $u_i\mathbb{S}_i = -u_ip_i\mathbb{I} + \eps \mathbb{T}_i$, where $p_i$ is the phase-specific pressure, $\mathbb{I}$ the unit matrix, and $\eps \mathbb{T}_i$ the viscous stress. The pressure $p_i$ is the sum of the overall pressure $p$ that is common to all mixture components and the intraphase pressure $q_i(u)$, 
\begin{align*}
  p_i = p + q_i(u), \quad q_i(u) = \sum_{j=0}^n q_{ij}u_j,
  \quad i=0,\ldots,n.
\end{align*}
Compared to \cite[(11)]{LKBJS06}, we allow for the self-pressure $q_{ii}$. The interaction force $F_i$ is the sum of the interphase forces $f_{ij}$ and the viscous drag forces $g_{ij}$,
\begin{align*}
  F_i = \sum_{j\neq i}(f_{ij}+g_{ij}), \quad
  f_{ij} = p_{ij} u_j\na u_i - p_{ji} u_i\na u_j, \quad
  g_{ij} = -k_{ij} u_iu_j(v_i-v_j),
\end{align*}
where $p_{ij}$ and $k_{ji}$ are nonnegative parameters. The interphase pressure $p_{ij}$ is the sum of the overall pressure $p$ and some extra pressure $r_{ij}$ due to traction between the phases $i$ and $j$, i.e.
\begin{align*}
  p_{ij} = p+r_{ij}, \quad i,j=0,\ldots,n.
\end{align*}
For convenience, we set $r_{ii}=0$ and $k_{ii}=0$ for $i=0,\ldots,n$. If $(p_{ij})$ and $(k_{ij})$ are symmetric, the total force $\sum_{i=0}^n F_i$ vanishes, which means that there are no external forces. Possible external forces can be modeled by nonsymmetric coefficients $p_{ij}$. In contrast to \cite{LKBJS06}, we allow for $q_{ij}\neq r_{ij}$. 

\subsection{Derivation of the cross-diffusion equations}

In the (formal) limit $\eps\to 0$ in \eqref{2.mom}, the stress tensor $u_i \mathbb{S}_i$ simplifies to $-u_ip_i\mathbb{I}$, and the momentum balance becomes
\begin{align*}
  0 &= -\na(u_ip_i) + F_i = -\na(u_ip+u_iq_i(u)) \\
  &\phantom{xx}
  + \sum_{j=0}^n \big((p+r_{ij})u_j\na u_i - (p+r_{ji})u_i\na u_j\big)
  - \sum_{j=0}^n k_{ij}u_iu_j(v_i-v_j).
\end{align*}

In the following, we suppose that $(r_{ij})$ and $(k_{ij})$ are symmetric. Taking into account that $\sum_{j=0}^n u_j=1$, the second term on the right-hand side can be simplified:
\begin{align*}
  \sum_{j=0}^n\big((p+r_{ij})u_j\na u_i - (p+r_{ji})u_i\na u_j\big)
  = p\na u_i + r_i(u)\na u_i - u_i\na r_i(u),
\end{align*}
where we have set $r_i(u)=\sum_{j=0}^n r_{ij}u_j$ and used $r_{ij}=r_{ji}$. This shows that
\begin{align}\label{2.nap}
  0 &= -\na(u_ip_i) + F_i \\
  &= -u_i\na p - \na(u_iq_i(u)) + (r_i(u)\na u_i-u_i\na r_i(u)) 
  - \sum_{j=0}^n k_{ij}u_iu_j(v_i-v_j). \nonumber 
\end{align}
The overall mixture pressure $p$ can be interpreted as a Lagrange multiplier associated to the constraint $\sum_{i=0}^n u_i=1$. We can remove $\na p$ by summing \eqref{2.nap} over $i=0,\ldots,n$. Taking into account that $(k_{ij})$ is symmetric, we find that
\begin{align*}
  0 = -\na p - \sum_{i=0}^n\na(u_iq_i(u))
  + \sum_{i=0}^n(r_i(u)\na u_i-u_i\na r_i(u)),
\end{align*}
which provides an expression for $\na p$. We replace $\na p$ in \eqref{2.nap} by this expression:
\begin{align}\label{2.kij}
  \sum_{j=0}^n k_{ij}u_iu_j(v_i-v_j)
  &= u_i\bigg(\sum_{j=0}^n\na(u_jq_j(u))
  - \sum_{j=0}^n(r_j(u)\na u_j-u_j\na r_j(u))\bigg) \\
  &\phantom{xx}- \na(u_iq_i(u)) + (r_i(u)\na u_i-u_i\na r_i(u)). \nonumber 
\end{align}
This yields equations \eqref{1.mass}--\eqref{1.mom}. Let $k_{ij}=1$ for $i,j=0,\ldots,n$. Because of $\sum_{j=0}^n u_jv_j=0$, the left-hand side of \eqref{2.kij} becomes 
\begin{align*}
  \sum_{j=0}^n k_{ij}u_iu_j(v_i-v_j)
  = \sum_{j=0}^n\big(u_j(u_iv_i)-u_i(u_jv_j)\big) = u_iv_i,
\end{align*}
and we recover equations \eqref{1.samek}. 


\subsection{Entropy structure}\label{sec.ent}

We show that equations \eqref{1.samek} possess a Boltzmann and Rao entropy structure. The case of unequal coefficients $k_{ij}\neq 1$ is investigated in Proposition \ref{prop.perturb}. Recall the definitions of the Boltzmann and Rao entropy densities:
\begin{align}\label{3.ent}
  h_B(u) = \sum_{i=0}^n u_i(\log u_i-1), \quad 
  h_R(u) = \frac12\sum_{i,j=0}^n q_{ij} u_i u_j.
\end{align}

\begin{lemma}[Boltzmann entropy equality]\label{lem.hB}
Let $k_{ij}=1$ for all $i,j=0,\ldots,n$, $i\neq j$ and let $u$ be a positive smooth solution to \eqref{1.qr}--\eqref{1.samek}. Then
\begin{align*}
  \frac{d}{dt}\int_\Omega h_B(u)dx 
  + \sum_{i,j=0}^n\int_\Omega(q_{ij}+r_{ij})\na u_i\cdot\na u_j dx 
  + 4\sum_{i=0}^n\int_\Omega(q_{i}(u)-r_{i}(u))|\na\sqrt{u_i}|^2dx = 0.
\end{align*} 
In particular, $t\mapsto \int_\Omega h_B(u(t))dx$ is a Lyapunov functional if the matrix $(q_{ij}+r_{ij})_{i,j=0}^n$ is positive semidefinite and  $q_{ij}\ge r_{ij}$. 
\end{lemma}

\begin{proof}
We compute formally:
\begin{align*}
  \frac{d}{dt}&\int_\Omega h_B(u)dx 
  = \sum_{i=0}^n\int_\Omega \pa_t u_i\log u_i dx
  = \sum_{i=0}^n\int_\Omega u_iv_i\cdot\na\log u_i dx \\
  &= -\int_\Omega\bigg(\sum_{i=0}^n q_i(u)\frac{|\na u_i|^2}{u_i}
  + \sum_{i=0}^n \na q_i(u)\cdot\na u_i 
  - \sum_{j=0}^n\na(u_jq_j(u))\cdot\sum_{i=0}^n \na u_i\bigg)dx \\
  &\phantom{xx}+ \int_\Omega\bigg(\sum_{i=0}^n
  r_i(u)\frac{|\na u_i|^2}{u_i} - \sum_{i=0}^n\na r_i(u)\cdot\na u_i \\
  &\phantom{xx}
  - \sum_{j=0}^n(r_j(u)\na u_j-u_j\na r_j(u))\cdot\sum_{i=0}^n\na u_i
  \bigg)dx.
\end{align*}
Taking into account that $\sum_{i=0}^n\na u_i=0$ and inserting the definitions of $q_i(u)$ and $r_i(u)$, we obtain
\begin{align*}
  \frac{d}{dt}\int_\Omega h_B(u)dx  = -\sum_{i,j=0}^n\int_\Omega
  (q_{ij}+r_{ij})\na u_j\cdot\na u_i dx
  - \sum_{i=0}^n\int_\Omega(q_i(u)-r_i(u))\frac{|\na u_i|^2}{u_i}dx,
\end{align*}
which finishes the proof.
\end{proof}

If the drag coefficients vary, we can show that the system has a Rao entropy structure. The Rao entropy can be interpreted as the potential energy of the fluid mixture, and the associated energy dissipation originates from the drag effects.

\begin{lemma}[Rao entropy inequality]\label{lem.rao}
Let $q_{ij}=r_{ij}$ for $i,j=0,\ldots,n$ be symmetric and let $u$ be a smooth solution to \eqref{1.mass}--\eqref{1.bic} (the coefficients $k_{ij}$ may vary). Then
\begin{align*}
  \frac{d}{dt}\int_\Omega h_R(u)dx 
  + \frac14\sum_{i,j=0}^n\int_\Omega k_{ij}u_iu_j|v_i-v_j|^2 dx = 0.
\end{align*}
\end{lemma}

\begin{proof}
In case $q_{ij}=r_{ij}$, the fluxes become
\begin{align}\label{2.uv}
  \sum_{j=0}^n k_{ij}u_iu_j(v_i-v_j) 
  = -2u_i\na q_i(u) + 2u_i\sum_{j=0}^n u_j\na q_j(u),
  \quad i=0,\ldots,n.
\end{align}
Furthermore, the symmetry $q_{ij}=q_{ji}$ implies that $\pa h_R/\pa u_i=q_i(u)$, and we infer from \eqref{2.uv} and the symmetry of $k_{ij}$ that 
\begin{align*}
  \frac{d}{dt}\int_\Omega h_R(u)dx
  &= \sum_{i=0}^n\int_\Omega u_iv_i\cdot\na q_i(u)dx \\
  &= \sum_{i=0}^n\int_\Omega v_i\cdot
  \bigg(-\frac12\sum_{j=0}^n k_{ij}u_iu_j(v_i-v_j)
  + u_i\sum_{j=0}^n u_j\na q_j(u)\bigg)dx \\
  &= -\frac14\sum_{i,j=0}^n\int_\Omega k_{ij}u_iu_j(v_i-v_j)\cdot v_i dx
  -\frac14\sum_{i,j=0}^n\int_\Omega k_{ji}u_ju_i(v_j-v_i)\cdot v_j dx \\
  &\phantom{xx}+ \int_\Omega \sum_{i=0}^n 
  u_iv_i\cdot\sum_{j=0}^n u_j\na q_j(u)dx \\
  &= -\frac14\sum_{i,j=0}^n\int_\Omega k_{ij}u_iu_j|v_i-v_j|^2 dx
  + \int_\Omega \sum_{i=0}^n u_iv_i\cdot\sum_{j=0}^n u_j\na q_j(u)dx.
\end{align*}
As the sum over $i=0,\ldots,n$ in the last integral vanishes, the lemma follows.
\end{proof}

One may wonder whether the Boltzmann entropy also provides some estimates, like in \cite{CCDJ25}. This is indeed possible if the drag coefficients include a``degeneracy''; see Remark \ref{rem.comb}. 

\begin{lemma}[Combined Boltzmann and Rao entropies]\label{lem.comb}
Let $k_{ij} = 1 + k_{ij}^*\sqrt{u_iu_j}$ with symmetric $k_{ij}^*>0$, $q_{ij}=r_{ij}$, $q_{i0}=q_{0i}=0$ for $i,j=0,\ldots,n$, $(q_{ij})_{i,j=1}^n$ be positive definite (with smallest eigenvalue $\alpha>0$), and $u$ be a positive smooth solution to \eqref{1.mass}--\eqref{1.bic}. Then there exists $C>0$, depending on $k_{ij}^*$ and $q_{ij}$, such that
\begin{align*}
  \frac{d}{dt}\int_\Omega\big(h_B(u) + Ch_R(u)\big)dx
  + \alpha \sum_{i=0}^n\int_\Omega |\na u_i|^2 dx \le 0.
\end{align*}
\end{lemma}

\begin{proof}
The fluxes become in our situation
\begin{align*}
  u_iv_i + \sum_{j=0}^n k_{ij}^*(u_iu_j)^{3/2}(v_i-v_j) 
  = -2u_i\na q_i(u) + 2u_i\sum_{j=0}^n\int_\Omega u_j\na q_j(u),
\end{align*}
where $i=0,\ldots,n$. Inserting this expression in
\begin{align*}
  \frac{d}{dt}\int_\Omega h_B(u)dx
  = \sum_{i=0}^n\int_\Omega u_iv_i\cdot\na\log u_i dx
\end{align*}
yields that
\begin{align*}
  \frac{d}{dt}\int_\Omega h_B(u)dx
  &= -\sum_{i,j=0}^n\int_\Omega k_{ij}^*(u_iu_j)^{3/2}(v_i-v_j)
  \cdot\na\log u_i dx \\
  &\phantom{xx}- 2\sum_{i=0}^n\int_\Omega 
  \na q_i(u)\cdot\na u_i dx 
  + 2\int_\Omega\bigg(\sum_{i=0}^n\na u_i\bigg)
  \cdot\bigg(\sum_{j=0}^n u_j\na q_j(u)\bigg)dx \\
  &= -\frac12\sum_{i,j=0}^n\int_\Omega k_{ij}^*(u_iu_j)^{3/2}(v_i-v_j)
  \cdot\na(\log u_i-\log u_j)dx \\
  &\phantom{xx}- 2\sum_{i,j=0}^n\int_\Omega q_{ij}\na u_i\cdot\na u_j dx,
\end{align*}
where we used the symmetry of $k_{ij}^*$ and the fact that $\sum_{i=0}^n\na u_i=0$. Next, by Young's inequality with $\delta>0$, $q_{i0}=q_{0i}=0$, and the positive definiteness of $(q_{ij})_{i,j=1}^n$,
\begin{align*}
  \frac{d}{dt}\int_\Omega h_B(u)dx
  &\le -2\alpha\sum_{i=0}^n|\na u_i|^2 dx
  + \delta\sum_{i,j=0}^n\int_\Omega u_i^2u_j^2
  |\na(\log u_i-\log u_j)|^2 dx \\
  &\phantom{xx}+ C(\delta)\sum_{i,j=0}^n\int_\Omega 
  k_{ij}^*u_iu_j|v_i-v_j|^2 dx.
\end{align*}
Because of $u_i\le 1$, the second term on the right-hand side is estimated as
\begin{align*}
  \sum_{i,j=0}^n&\int_\Omega u_i^2u_j^2
  |\na(\log u_i-\log u_j)|^2 dx \\
  &= \sum_{i,j=0}^n\int_\Omega\big(u_j^2|\na u_i|^2 + u_i^2|\na u_j|^2
  - 2u_iu_j\na u_i\cdot\na u_j\big)dx
  \le C_0\sum_{i=0}^n\int_\Omega |\na u_i|^2 dx.
\end{align*}
Then, taking into account Lemma \ref{lem.rao} and choosing $\delta = \alpha/C_0$, we find that
\begin{align*}
  \frac{d}{dt}\int_\Omega\big(h_B(u) + C(\delta)h_R(u)\big)dx
  \le (-2\alpha + \delta C_0)\sum_{i=0}^n\int_\Omega |\na u_i|^2 dx
  \le -\alpha \sum_{i=0}^n\int_\Omega |\na u_i|^2 dx,
\end{align*}
which ends the proof.
\end{proof}

\begin{remark}\label{rem.comb}\rm
The previous lemma indicates that the pressure and drag forces are not compatible with each other, which may explain the difficulties in determining the existence of global weak solutions to \eqref{1.mass}--\eqref{1.bic} for general (symmetric) $k_{ij}$. Indeed, the drag forces yield a nondegenerate diffusion. For instance, choosing $q_i=1$ and $r_i=0$ such that $u_iv_i=-\na u_i$, we end up with the diffusion equation $\pa_t u_i = \Delta u_i$. The pressure forces, however, yield a quadratic porous-medium-type degeneracy like $u_i\na u_i$. If we replace $k_{ij}$ by $k_{ij}^*\sqrt{u_iu_j}$, the drag forces include a quadratic degeneracy, and we are able to estimate both pressure and drag forces. Notice that in contrast to Proposition \ref{prop.perturb}, the coefficients $k_{ij}$ do {\em not} need to be a perturbation of one. 
\end{remark}


\section{Mathematical analysis: equal drag coefficients}
\label{sec.proofs}

In this section, we prove Theorems \ref{thm.ex}--\ref{thm.wsu}.

\subsection{Proof of Theorem \ref{thm.ex}}\label{sec.ex}

The existence of global weak solutions to \eqref{1.qr}--\eqref{1.samek} follows from the boundedness-of-entropy method \cite[Theorem 2]{Jue15}. To this end, we need to verify the hypotheses of this theorem. The entropy density $h_B(u)=\sum_{i=0}^n u_i(\log u_i-1)$ with $u\in\dom$ and $u_0=1-\sum_{i=1}^n u_i$ is convex, satisfies $h_B\in C^2(\dom)$, and the derivative $h'_B:\dom\to\R^n$ is invertible on $\R^n$. This verifies Hypothesis (H1) in \cite{Jue15}. 

Since $q_0(u)=r_0(u)=0$, the sums in \eqref{1.samek} range from $i=1,\ldots,n$. This gives for the flux $J_i=u_iv_i$:
\begin{align*}
  J_i = -\sum_{j=1}^n A_{ij}(u)\na u_j\quad\mbox{for }i=1,\ldots,n,
\end{align*}
where for $i,j=1,\ldots,n$, 
\begin{align}\label{4.A}
  A_{ij}(u) = (q_i(u)-r_i(u))\delta_{ij}
  + u_i\bigg(q_{ij} - q_j(u) + r_{ij} + r_j(u) 
  - \sum_{\ell=1}^n u_\ell(q_{\ell j}+r_{\ell j})\bigg).
\end{align}
The Hessian of the Boltzmann entropy density becomes
\begin{align*}
  h''_B(u) = \frac{\delta_{ij}}{u_i} + \frac{1}{u_0}
  \quad\mbox{for }i,j=1,\ldots,n.
\end{align*}

\begin{lemma}\label{lem.hA}
Let the symmetric part of $(q_{ij}+r_{ij})_{i,j=1}^n$ be positive definite with smallest eigenvalue $\alpha>0$ and let $q_{ij}\ge r_{ij}$ for $i,j=1,\ldots,n$. Then
\begin{align*}
  z^Th_B''(u)A(u)z \ge \alpha|z|^2 \quad\mbox{for all }u\in\dom,\ 
  z\in\R^n.
\end{align*}
\end{lemma}

The lemma shows that Hypothesis (H2) in \cite{Jue15} is satisfied.

\begin{proof}
A computation shows that
\begin{align*}
  h''_B(u)A(u) = \frac{\delta_{ij}}{u_i}(q_i(u)-r_i(u))
  + q_{ij} + r_{ij}\quad\mbox{for }i,j=1,\ldots,n.
\end{align*}
We infer from our assumptions that $q_i(u)-r_i(u)\ge 0$, which yields 
\begin{align*}
  z^T h''_B(u)A(u)z = \sum_{i=1}^n(q_i(u)-r_i(u))\frac{z_i^2}{u_i}
  + \sum_{i,j=1}^n(q_{ij}+r_{ij})z_iz_j
  \ge \alpha\sum_{i=1}^n z_i^2
\end{align*}
for $z\in\R^n$, which proves the lemma.
\end{proof}

Since we do not consider reaction terms, Hypothesis (H3) in \cite{Jue15} is void. Hence, all hypotheses of \cite[Theorem 2]{Jue15} are satisfied, and we conclude the existence of a global bounded weak solution to \eqref{1.qr}--\eqref{1.samek}. The entropy inequality \eqref{1.ent} is a consequence of Lemma \ref{lem.hB}. This proves Theorem \ref{thm.ex}. 

\begin{remark}[Nonconstant coefficients]\label{rem.qr}\rm
The coefficients $q_{ij}$ and $r_{ij}$ may depend on $u$ as long as they are continuous on $\overline{\dom}$ and the matrix $(q_{ij}+r_{ij})_{i,j=1}^n$ is positive definite uniformly in $\dom$. Indeed, the continuity is required to derive the continuity of the fixed-point operator in \cite{Jue15}, while the uniform positive definiteness implies the validity of Hypothesis (H2) in \cite{Jue15}. We obtain the existence of global bounded weak solutions to \eqref{1.qr}--\eqref{1.samek} also in this situation. 
\qed\end{remark}


\subsection{Proof of Theorem \ref{thm.decay}}\label{sec.decay}

Let $u_i^\infty:=|\Omega|^{-1}\int_\Omega u_i^0 dx$ for $i=0,\ldots,n$. Using \eqref{3.ent}, we introduce the relative Boltzmann entropy
\begin{align*}
  H_B(u|u^\infty) = \int_\Omega\big(h_B(u)-h_B(u^\infty)-h'(u^\infty)\cdot(u-u^\infty)
  \big)dx = \sum_{i=0}^n\int_\Omega u_i\log\frac{u_i}{u_i^\infty}dx,
\end{align*}
where the last step follows after an elementary computation. By the chain rule (see, e.g., \cite[Lemma 1.5]{AlLu83}) and since $u^\infty$ is constant,
\begin{align*}
  \frac{d}{dt}H_B(u|u^\infty)
  &= \big\langle
  \pa_t u_i,h'_B(u)-h_B'(u^\infty)\big\rangle
  = \langle\pa_t u,h'_B(u)\rangle ds \\
  &= -\sum_{i,j=1}^n\int_\Omega (h_B''(u)A(u))_{ij}
  \na u_i\cdot\na u_j dx
  \le -\alpha\sum_{i=1}^n\int_\Omega|\na u_i|^2 dx,
\end{align*}
and we used Lemma \ref{lem.hA} in the last step. Observing that
\begin{align}\label{3.n}
  \sum_{i=0}^n\int_\Omega|\na u_i|^2 dx
  = \sum_{i=1}^n\int_\Omega|\na u_i|^2 dx
  + \bigg|\sum_{i=1}^n\int_\Omega\na u_i dx\bigg|^2
  \le (n+1)\sum_{i=1}^n\int_\Omega|\na u_i|^2 dx,
\end{align}
we obtain
\begin{align*}
  \frac{d}{dt}H_B(u|u^\infty) \le -\frac{\alpha}{n+1}
  \sum_{i=0}^n\int_\Omega|\na u_i|^2 dx.
\end{align*}
We apply the following variant of the logarithmic Sobolev inequality \cite[Theorem 1]{AbLe25}:
\begin{align*}
  \int_\Omega u_i\log\frac{u_i}{u_i^\infty}dx
  \le \frac{C_S}{u_i^\infty}\int_\Omega|\na u_i|^2 dx,
\end{align*}
where $C_S>0$ only depends on $\Omega$ and the space dimension $d$. Summing over $i=0,\ldots,n$, this yields 
\begin{align*}
  \frac{d}{dt}H_B(u|u^\infty) \le -\frac{\alpha}{C_S(n+1)}
  \Big(\min_{i=0,\ldots,n}u_i^\infty\Big) 
  H_B(u|u^\infty) \quad\mbox{for }t>0.
\end{align*}
We infer from Gronwall's lemma that
\begin{align*}
  H_B(u(t)|u^\infty) \le H_B(u^0|u^\infty)e^{-c t},
\end{align*}
where $c=\alpha (C_S(n+1))^{-1}\min_{i=0,\ldots,n}u_i^\infty>0$. The proof follows after an application of the Csisz\'ar--Kullback--Pinsker inequality with $\lambda=c/2$ \cite[Theorem A.2]{Jue16}.

Notice that our assumption that the first row and column of the matrix $(q_{ij}+r_{ij})_{i,j=0}^n$ vanish leads to a reduced decay rate; we obtain a rate with factor $\alpha/(n+1)$ instead of $\alpha$. This is not surprising. By imposing our assumption, we lose the decay information for $u_0$. This information must be recovered indirectly from $\na u_0=-\sum_{i=1}^n\na u_i$, and consequently the resulting decay rate is smaller.


\subsection{Proof of Theorem \ref{thm.wsu}}\label{sec.wsu}

Let $u$ be a weak solution and $\bar{u}$ be a strong solution to \eqref{1.bic}, \eqref{1.regBT}. The difference satisfies the equation
\begin{align*}
  \pa_t(u_i-\bar{u}_i) &= \eta\Delta(u_i-\bar{u}_i)
  + 2\diver\bigg(u_i\na q_i(u) - \bar{u}_i\na q_i(\bar{u}) \\
  &\phantom{xx}- u_i\sum_{j=1}^n
  \big(u_j\na q_j(u)-\bar{u}_j\na q_j(\bar{u})\big)\bigg), \quad
  i=0,\ldots,n.
\end{align*}
We differentiate the relative Rao entropy
\begin{align*}
  H_R(u|\bar{u}) = \frac{1}{2}\sum_{i,j=0}^n\int_\Omega q_{ij}
  (u_i-\bar{u}_i)(u_j-\bar{u}_j) dx
\end{align*}
with respect to time and insert the equation satisfied by $u_i-\bar{u}_i$:
\begin{align*}
  &\frac{d}{dt}H_R(u|\bar{u}) 
  = \sum_{i,j=0}^n q_{ij}\langle\pa_t(u_i-\bar{u}_i),
  u_j-\bar{u}_j\rangle \\
  &= \sum_{i=0}^n\langle\pa_t(u_i-\bar{u}_i),q_i(u)-q_i(\bar{u})
  \rangle = I_1+ 2I_2 + 2I_3, \\
\end{align*}
where 
\begin{align*}
  I_1 &= -\eta\sum_{i,j=0}^n\int_\Omega q_{ij}\na(u_i-\bar{u}_i)
  \cdot\na(u_j-\bar{u}_j)dx, \\
  I_2 &= -2\sum_{i=0}^n\int_\Omega
  (u_i\na q_i(u)-\bar{u}_i\na q_i(\bar{u}))
  \cdot\na(q_i(u)-q_i(\bar{u}))\big)dx, \\
  I_3 &= 2\sum_{i,j=0}^n\int_\Omega
  \big(u_iu_j\na q_j(u) - \bar{u}_i\bar{u}_j\na q_j(\bar{u})\big)
  \cdot\na(q_i(u)-q_i(\bar{u}))dx.
\end{align*}
In fact, all sums run from $i=1,\ldots,n$ and $j=1,\ldots,n$, since $q_0(u)=q_0(\bar{u})=0$ by assumption. (The proof also holds without this assumption, using inequality \eqref{3.n}.) By the positive definiteness of $(q_{ij})$,
\begin{align*}
  I_1 \le -\alpha\eta\sum_{i=1}^n\int_\Omega|\na(u_i-\bar{u}_i)|^2 dx.
\end{align*}
We add and subtract some terms in $I_2$ and $I_3$, leading to
\begin{align*}
  I_2 &= -\sum_{i=1}^n\int_\Omega u_i|\na(q_i(u)-q_i(\bar{u}))|^2dx \\
  &\phantom{xx}- \sum_{i=1}^n\int_\Omega (u_i-\bar{u}_i)\na q_i(\bar{u})
  \cdot\na(q_i(u)-q_i(\bar{u}))dx =: I_{21}+I_{22}, \\
  I_3 &= \sum_{i,j=1}^n\int_\Omega u_iu_j
  \na(q_i(u)-q_i(\bar{u}))\cdot\na(q_j(u)-q_j(\bar{u}))dx \\
  &\phantom{xx}+ \sum_{i,j=1}^n\int_\Omega (u_iu_j-\bar{u}_i\bar{u}_j)
  \na q_j(\bar{u})\cdot\na(q_i(u)-q_i(\bar{u}))dx =: I_{31} + I_{32}.
\end{align*}
It follows from H\"older's inequality for sums that $I_{21}+I_{31}\le 0$, since
\begin{align*}
  I_{31} &= \int_\Omega\bigg(\sum_{i=1}^n u_i
  \na(q_i(u)-q_i(\bar{u}))\bigg)^2 dx \\
  &\le \int_\Omega\bigg(\sum_{j=1}^n u_j\bigg)
  \bigg(\sum_{i=1}^n u_i|\na(q_i(u)-q_i(\bar{u}))|^2\bigg) dx = -I_{21}.
\end{align*}
Next, using Young's inequality with $\delta>0$ and the bound $u_i\le 1$, 
\begin{align*}
  I_{22} &\le \delta\sum_{i=1}^n\int_\Omega
  |\na(q_i(u)-q_i(\bar{u}))|^2 dx
  + C_1(\delta)\sum_{i=1}^n\int_\Omega|u_i-\bar{u}_i|^2 dx \\
  &\le C_2\delta\sum_{i=1}^n\int_\Omega|\na(u_i-\bar{u}_i)|^2 dx
  + C_1(\delta)\sum_{i=1}^n\int_\Omega|u_i-\bar{u}_i|^2 dx,
\end{align*}
where $C_1(\delta)>0$ depends on $\delta$ and the $L^\infty(\Omega_T)$ norm of $\max_i q_i(\bar{u})$ and $C_2>0$ depends on $\max_{i,j}q_{ij}$. In a similar way,
\begin{align*}
  I_{32} \le C_2\delta\sum_{i=1}^n\int_\Omega|\na(u_i-\bar{u}_i)|^2 dx
  + C_3(\delta)\sum_{i=1}^n\int_\Omega|u_i-\bar{u}_i|^2 dx.
\end{align*} 
This yields
\begin{align*}
  \frac{d}{dt}H_R(u|\bar{u}) &\le (\alpha\eta - 4C_2\delta)
  \sum_{i=1}^n\int_\Omega|\na(u_i-\bar{u}_i)|^2 dx
  + 2(C_1(\delta)+C_3(\delta))\sum_{i=1}^n
  \int_\Omega|u_i-\bar{u}_i|^2 dx.
\end{align*}
Taking into account that
\begin{align*}
  H_R(u|\bar{u}) \ge \frac{\alpha}{2}\sum_{i=1}^n\int_\Omega
  |u_i-\bar{u}_i|^2 dx
\end{align*}
and choosing $0<\delta\le \alpha\eta/(4C_2)$, we conclude that
\begin{align*}
  \frac{d}{dt}H_R(u|\bar{u}) \le \frac{4}{\alpha}
  (C_1(\delta)+C_3(\delta))H_R(u|\bar{u}).
\end{align*}
We deduce from the fact that $u$ and $\bar{u}$ have the same initial data that $H_R(u(0)|\bar{u}(0))=0$. Therefore, it follows from Gronwall's lemma that $H_R(u(t)|\bar{u}(t))=0$ for $t>0$ and consequently $u_i(t)=\bar{u}_i(t)$ in $\Omega$ for $t>0$ and $i=1,\ldots,n$. The theorem is proved.


\subsection{Tumor-growth model of Jackson and Byrne}\label{sec.tg}

The tumor-growth model of \cite{JaBy02} is given by \eqref{1.qr}--\eqref{1.samek} with $n=2$, $r_{ij}=0$, and
\begin{align*}
  q_{11} = 1, \quad q_{12} = 0, \quad q_{21} = \beta\theta u_2,
  \quad q_{22} = \beta,
\end{align*}
where $\beta>0$ and $\theta>0$. The matrix $(q_{ij})$ is positive definite if $\theta<2/\sqrt{\beta}$, since 
\begin{align*}
  q_{11}z_1^2 + (q_{12}+q_{21})z_1z_2 + q_{22}z_2^2
  = z_1^2 + \beta\theta u_2z_1z_2 + \beta z_2^2> 0
\end{align*}
for all $z\in\R^2$, $z\neq 0$ if and only if $4>\beta\theta^2u_2^2$, which is the case if $\theta<2/\sqrt{\beta}$. If $\theta$ is sufficiently large, the problem may be ill-posed, as shown in the following lemma.

\begin{lemma}
Let $\beta>0$ and let $\theta>0$ be sufficiently large. Then there exists $(u_1,u_2)\in\dom$, satisfying $u_1\in(0,1/3)$, $u_2\in(2/3,1)$ such that at least one eigenvalue of $A(u)$ has a negative real part.
\end{lemma}

\begin{proof}
We recall that the eigenvalues of the $2\times 2$ matrix $A(u)$ have positive real parts if $\det A(u)>0$ and $\operatorname{tr}A(u)>0$. A computation shows that for $u\in\dom$,
\begin{align*}
  \operatorname{tr}A(u) &= \beta\theta u_1 u_2(2-3u_2) 
  + 2\beta u_2(1-u_2) + 2u_1(1-u_1), \\
  \det A(u) &= 4\beta(1+\theta u_1)(1-u_1-u_2) > 0.
\end{align*}
We have $\operatorname{tr}A(u)>0$ if $u_2<2/3$. Thus, to obtain negative real parts, we need to assume that $u_2>2/3$. Let $0<\eps<1/6$, $u_2=2/3+\eps$, and $u_1=1/3-2\eps$. Then $(u_1,u_2)\in\dom$ and
\begin{align*}
  \operatorname{tr}A(u) = -3\eps\beta\theta\bigg(\frac13 - 2\eps\bigg)
  \bigg(\frac23 + \eps\bigg) + 2\beta\bigg(\frac23 + \eps\bigg)
  \bigg(\frac13 - \eps\bigg) + 2\bigg(\frac13 - 2\eps\bigg)
  \bigg(\frac23 + 2\eps\bigg).
\end{align*}
For any given $\beta>0$, there exists $\theta>0$ (depending on $\eps$) such that $\operatorname{tr}A(u)<0$. We infer that the real part of $\lambda_2$ is negative.
\end{proof}


\section{Mathematical analysis: different drag coefficients} \label{sec.unsamek}

Let $K(u)$ be the matrix that satisfies
\begin{align*}
  \sum_{j=0}^n k_{ij}u_iu_j(v_i-v_j)
  = \bigg(\sum_{j=0}^n k_{ij}u_j\bigg)J_i
  - \sum_{j=0}^n k_{ij}u_iJ_j = (K(u)J)_i, \quad i=1,\ldots,n.
\end{align*}
Observing that $J_0=-\sum_{i=1}^n J_i$ and $u_0=1-\sum_{i=1}^n u_i$, we compute
\begin{align}\label{5.K}
  K_{ij}(u) = \begin{cases}
  \sum_{\ell=0}^n k_{i\ell}u_\ell + k_{i0}u_i &\mbox{for }i=j, \\
  (k_{i0}-k_{ij})u_i &\mbox{for }i\neq j,
  \end{cases}
\end{align}
recalling that $k_{ii}=0$. Then we can formulate equations \eqref{1.mass}--\eqref{1.mom} as
\begin{align*}
  \pa_t u_i + \diver J_i = 0, \quad K(u)J = -A(u)\na u.
\end{align*}
It is shown in \cite[Lemma 2.3]{JuSt13} that $K(u)$ is invertible for $u\in\dom$. Therefore,
\begin{align*}
  \pa_t u - \diver(K(u)^{-1}A(u)\na u) = 0.
\end{align*}
It turns out that the diffusion matrix $K(u)^{-1}A(u)$ is positively stable, i.e., the eigenvalues have positive real parts.

\begin{proposition}\label{prop.posstab}
Let the assumptions of Theorem \ref{thm.ex} hold. Then $K(u)^{-1}A(u)$ is positively stable. 
\end{proposition}

\begin{proof}
It is proved in \cite[Lemma 2.4]{JuSt13} that the matrix $K(u)^{-1}h_B''(u)^{-1}$ is symmetric positive definite for $u\in\dom$. The proof of Theorem \ref{thm.ex} implies that $h_B''(u)A(u)$ is symmetric positive definite. By \cite[Prop.~2]{ChJu21}, the product $K^{-1}(u)A(u) = (K^{-1}(u)h_B''(u)^{-1})(h_B''(u)A(u))$ is positively stable (but generally neither symmetric nor positive definite). 
\end{proof}

\begin{proposition}\label{prop.perturb}
Let the assumptions of Theorem \ref{thm.ex} hold. Let $ k_{ij} = 1 + \eps k_{ij}^*u_iu_j$ for symmetric numbers $k_{ij}^*>0$ and $\eps>0$, where $i,j=1,\ldots,n$. There exists $\eps_0>0$ such that for $0<\eps\le\eps_0$, the matrix $h_B''(u)K(u)^{-1}A(u)$ is positive definite for $u\in\dom$. In particular, equations \eqref{1.mass}--\eqref{1.bic} possesses a Boltzmann entropy structure and we conclude the existence of global bounded weak solutions to \eqref{1.mass}--\eqref{1.bic}.
\end{proposition}

\begin{proof}
We compute for $i=1,\ldots,n$,
\begin{align*}
  \sum_{j=0}^n k_{ij}u_iu_j(v_i-v_j)
  = J_i + \eps\bigg(\sum_{j=0}^n k_{ij}^*u_iu_j^2\bigg)J_i
  - \eps\sum_{j=0}^n k_{ij}^*u_i^2u_j J_j = (K(u)J)_i,
\end{align*}
where $J_i=u_iv_i$, $K(u) = \mathbb{I} + \eps K^*(u)$, $\mathbb{I}$ is the unit matrix in $\R^{n\times n}$, and
\begin{align*}
  K_{ij}^*(u) = \begin{cases}
  \sum_{\ell=0}^n k_{i\ell}u_iu_\ell^2 + k_{i0}u_0u_i^2
  &\mbox{for }i=j, \\
  (k_{i0}u_0-k_{ij}u_j)u_i^2 &\mbox{for }i\neq j.
  \end{cases}
\end{align*}
Then the entries of the product $h_B''(u)K^*(u)$ are polynomials in $u$, and there exists $C_1>0$ such that $\|h_B''(u)K^*(u)\|\le C_1$ for $u\in\dom$, where $\|\cdot\|$ is the Frobenius matrix norm. (Observe that we cannot achieve the boundedness if we perturb $k_{ij}$ by $\eps k_{ij}^*$; we need the factor $u_iu_j$.) Also the entries of $A(u)$ and $K^*(u)$ are polynomials, so there exists $C_2>0$ such that $\|A(u)\|\le C_2$ and $\|K^*(u)\|\le C_2$ for $u\in\dom$. Thus, choosing $\eps\le (2C_2)^{-1}$, the matrix $K(u)$ is invertible and
\begin{align*}
  (\mathbb{I}+\eps K^*(u))^{-1} 
  = \sum_{\ell=0}^\infty\eps^\ell K^*(u)^\ell
  = \mathbb{I} 
  + \eps K^*(u)\sum_{\ell=0}^\infty\eps^{\ell} K^*(u)^{\ell} .
\end{align*}
We compute for $z\in\R^n$,
\begin{align*}
  & z^Th_B''(u)K(u)^{-1}A(u)z = z^Th_B''(u)A(u)z + \eps z^TR(u)z, \\
  & \mbox{where }
  R(u) := h_B''(u)K^*(u)\bigg(\sum_{\ell=0}^\infty
  \eps^\ell K^*(u)^{\ell}\bigg)A(u).
\end{align*}
We estimate the remainder matrix $R(u)$ for $u\in\dom$:
\begin{align*}
  \|R(u)\| &\le \|h''_B(u)K^*(u)\|
  \bigg\|\sum_{\ell=0}^\infty\eps^\ell K^*(u)^\ell\bigg\| \|A(u)\| \\
  &\le \frac{\|h_B''(u)K^*(u)\|\|A(u)\|}{1 - \eps\|K^*(u)\|}
  \le 2C_1C_2 < \infty
\end{align*}
using $\eps\|K^*(u)\|\le 1/2$ in the last step. Choosing additionally $\eps\le \alpha(4C_1C_2)^{-1}$, it follows from the positive definiteness of $h''_B(u)A(u)$ that 
\begin{align*}
  z^Th_B''(u)K(u)^{-1}A(u)z \ge \alpha|z|^2 - \eps\|R(u)\||z|^2
  \le (\alpha - 2\eps C_1C_2)|z|^2 \le \frac{\alpha}{2}|z|^2.
\end{align*}
This finishes the proof. 
\end{proof}

We can allow for coefficients $k_{ij}$ that are not close to each other; we show this for the case $n=2$. 

\begin{proposition}\label{prop.G}
Let $n=2$, $k_{ij}>0$, $q_{ij}>0$, and $r_1=r_2=0$. Then the matrix $G(u):=h_B''(u)K(u)^{-1}A(u)$ is positive definite uniformly in $u\in\dom$ if:
\begin{itemize}
\item Case $k_{01}<k_{12}<k_{02}$ and $k_{02}\le k_{01}+k_{12}$:
\begin{equation}\label{4.case1}
\begin{aligned}
  & k_{12}q_{12} + (k_{01}+k_{12}-k_{02})q_{21}
  > 2(k_{02}-k_{12})q_{22}, \\
  & 2(k_{01}+k_{12})q_{11} > k_{02}q_{12} + (2k_{02}-k_{12})q_{21}, \\
  & 4k_{01}q_{22} > 2(k_{12}-k_{01})q_{11} + (k_{02}+k_{12})q_{12}
  + (k_{01}+k_{12})q_{21};
\end{aligned}
\end{equation}
\item Case $k_{02}<k_{01}<k_{12}$: 
\begin{align*}
  & 4k_{02} > k_{02}q_{12} + 2(k_{01}+k_{12})q_{21} 
  + 2(k_{12}-k_{02})q_{22}, \\
  & 2\min\{2k_{01},k_{02}+k_{12}\}q_{22} > 2(k_{12}-k_{01})q_{11}
  + (k_{02}+k_{12})q_{12} + (k_{01}+k_{12})q_{21};
\end{align*}
\item Case $k_{12}<k_{01}<k_{02}$:
\begin{align*}
  & k_{12}q_{12} + (k_{01}+k_{12}-k_{02})q_{21} > 2(k_{01}-k_{12})q_{11}
  + 2(k_{02}-k_{12})q_{22}, \\
  & 2(k_{01}+k_{12})q_{11} > k_{02}q_{12} + (2k_{02}-k_{12})q_{21}, \\
  & 2\min\{2k_{01},k_{02}+k_{12}\}q_{22} > k_{02}q_{12}
  + (k_{01}+k_{12})q_{21}.
\end{align*}
\end{itemize}
The matrix $G(u)$ is uniformly positive definite also in the cases
\begin{align*}
  k_{02}<k_{12}<k_{01}\mbox{ and }k_{01}\le k_{02}+k_{12}; \quad
  k_{02}<k_{01}<k_{12}; \quad k_{12}<k_{02}<k_{01},
\end{align*}
if we swap $k_{01}\leftrightarrow k_{02}$, $q_{12}\leftrightarrow q_{21}$, and $q_{11}\leftrightarrow q_{22}$ in the stated inequalities. 
\end{proposition}

We prove in fact a stronger result: There exists $c>0$ such that for all $z\in\R^2$,
\begin{align*}
  z^TG(u)z \ge c\sum_{i=0}^2 z_i^2, \quad\mbox{where }z_0+z_1+z_2=0.
\end{align*}
Choosing $z_i=\na u_i$, this gives a gradient bound for $u_i$ for $i=0,1,2$. The proof is based on elementary, but tedious estimations; see Appendix \ref{app}. Generally, the matrix $h_B''(u)K(u)^{-1}A(u)$ is {\em not} positive definite; see the following counterexample.

\begin{example}[Counterexample to entropy structure]\rm
Let $n=2$, $q_{12}=q_{21}$, and $r_1=r_2=0$. The matrix $K(u)$ can be computed explicitly:
\begin{align*}
  K(u) = \begin{pmatrix}
  k_{01} + (k_{12}-k_{01})u_2 & -(k_{12}-k_{01})u_1 \\
  -(k_{12}-k_{02})u_2 & k_{02} + (k_{12}-k_{02})u_1
  \end{pmatrix}.
\end{align*}
The inverse of $K(u)$ equals
\begin{align}\label{4.K1}
  K(u)^{-1} &= \frac{1}{\kappa(u)}\begin{pmatrix}
  k_{02} + (k_{12}-k_{02})u_1 & (k_{12}-k_{01})u_1 \\
  (k_{12}-k_{02})u_2 & k_{01} + (k_{12}-k_{01})u_2
  \end{pmatrix},
\end{align}
where $\kappa(u) = k_{01}k_{02}u_0 + k_{01}k_{12}u_1 + k_{02}k_{12}u_2>0$ for $u\in\dom$. The matrix $A(u)$ equals
\begin{align}\label{4.Au}
  A(u) = \begin{pmatrix}
  2q_{11}u_1(1-u_1) + q_{12}u_2(1-2u_1) 
  & q_{12}u_1(1-2u_1) - 2q_{22}u_1u_2 \\
  q_{12}u_2(1-2u_2) - 2q_{11}u_1u_2 
  & 2q_{22}(1-u_2) + q_{12}u_1(1-2u_2)
  \end{pmatrix}.
\end{align}
We choose $q_{11}=q_{22}=1$, $q_{12}=10$, $k_{01}=k_{02}=1$, $k_{12}=10$. Then, with $G(u) := h_B''(u)K(u)^{-1}A(u)$,
\begin{align*}
  G(u) = \gamma(u)\begin{pmatrix}
  (9u_1^2 + (90u_2+1)u_1 + 5u_2)u_2 & (90u_1 + 9u_2 + 5)u_1u_2 \\
  (9u_1 + 90u_2 + 5)u_1u_2 & (9u_2^2 + (90u_1+1)u_2 + 5u_1)u_1
  \end{pmatrix},
\end{align*}
where $\gamma(u)=2/(1+9u_1+9u_2)$. The symmetric part of this matrix has a negative eigenvalue if, for instance, $u\in Q\cup Q^T$, where $Q=[0.1,0.25]\times[0.55,0.75]$. Thus, $h_B''(u)K(u)^{-1}A(u)$ is not positive definite on $\dom$.
\qed\end{example}


\section{Numerical experiments}\label{sec.num}

We present some numerical experiments in one space dimension.

\subsection{Implementation of the scheme}

We discretize equations \eqref{1.samek} with $n=2$ and $r_i=0$ in one space dimension using finite volumes in space and the implicit Euler scheme in time. The nonlinear discrete system is solved by Newton's method. Convergence is reached when the residual norm falls below the threshold $10^{-10}$. The Jacobian is assembled analytically and mildly regularized (by adding $\eps=10^{-12}$ to the diagonal); a backtracking line search is employed for robustness. To ensure the physical constraints $u_i\ge 0$ for $i=0,1,2$, we have implemented a projection onto the simplex $\dom$. We found in all experiments that the discrete solution stays within the set $\dom$ without requiring any projection step during the iterations.

We have used a uniform spatial mesh on $\Omega=(0,1)$ consisting of $N=600$ control volumes, corresponding to a spatial step size $\Delta x\approx 1.67\cdot 10^{-3}$. The time integration is performed up to the final time $T=6$ using a constant time step size $\Delta t=10^{-3}$. The initial data represent a smooth segregation of phases, and it is the same for all simulations:
\begin{align*}
  u_1^0(x) = C_0
  \bigg(1 + \tanh\bigg(\frac{x_0 - x}{\eta}\bigg) \bigg) + \eps_0, \quad
  u_2^0(x) = C_0 
  \bigg(1 - \tanh\bigg(\frac{x_0 - x}{\eta}\bigg) \bigg),
\end{align*}
where $C_0=0.5$, $x_0=0.1$, $\eta=0.05$, and $\eps_0=0.01$. We have computed the numerical convergence rate when $\Delta x\to 0$ (by comparing to a reference solution on a very fine mesh with 2500 cells), to verify our implementation. Figure \ref{fig.rate} shows that the rate for the $L^1(\Omega)$ error at final time is approximately 1.24 for the tumor-growth model, which is presented in the following.

\begin{figure}[htb]
\includegraphics[width=70mm]{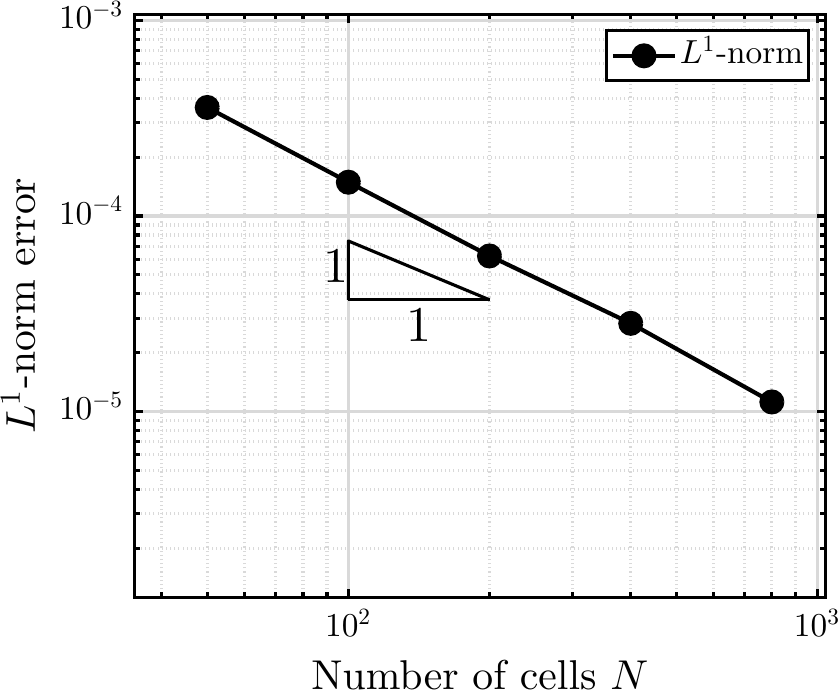}  
\caption{Convergence rate in the discrete $L^1$ norm for the tumor-growth model ($\beta=0.0015$, $\theta=100$).}
\label{fig.rate}
\end{figure}

\subsection{Tumor-growth model of Jackson and Byrne}

The choice
\begin{align*}
  q_{11} = \beta_c, \quad q_{12} = 0, \quad  
  q_{21} = \beta_m\theta u_2, \quad q_{22} = \beta_m
\end{align*}
correspond to the original model of \cite{JaBy02}. In the introduction, we have presented the special case $\beta:=\beta_c=\beta_m>0$. The entropy method is applicable if $\theta < \theta^* := 4\sqrt{\beta_c/\beta_m}$ \cite{JuSt12}. We choose $\beta_c=0.2$ and $\beta_m=0.0015$, which yields the threshold $\theta^*\approx 46.2$. Figure \ref{fig.tgm} shows the volume fractions of the tumor cells and the extracellular matrix (ECM) at various times for the subcritical parameter $\theta=30$ and the supercritical value $\theta=1000$. We observe peaks in the ECM, which move in time from left to right. This means that the tumor cells penetrate the surrounding ECM. The numerical results coincide with those from \cite[Sec.~5]{JuSt12}, confirming our numerical implementation.

\begin{figure}[htb]
\includegraphics[width=120mm]{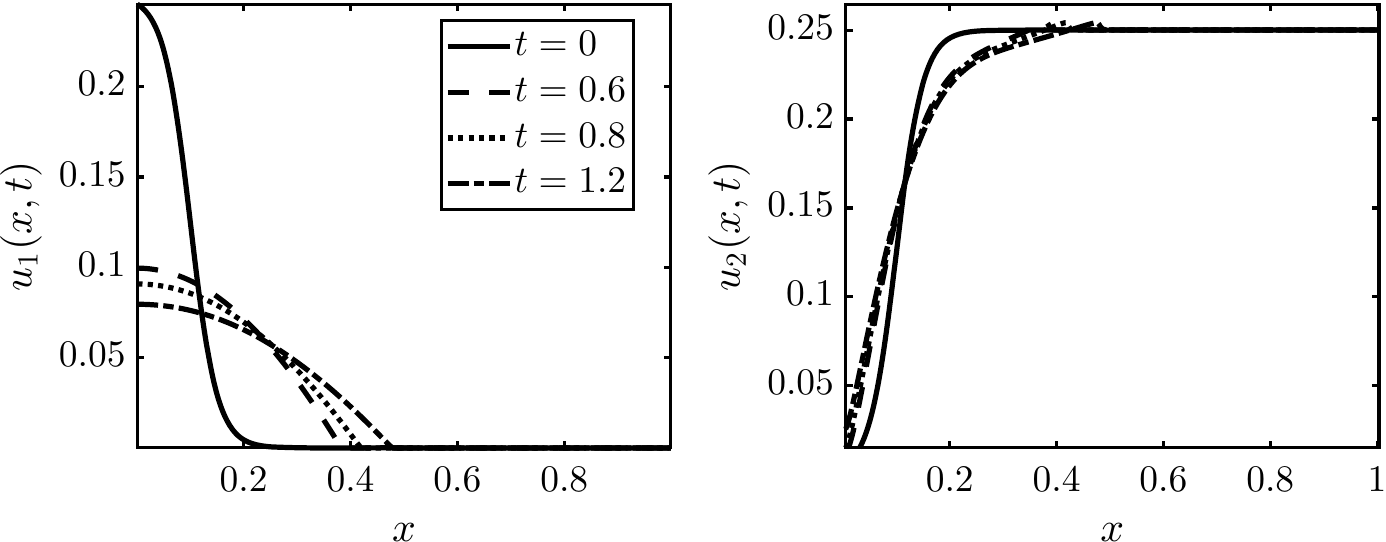} 
\includegraphics[width=120mm]{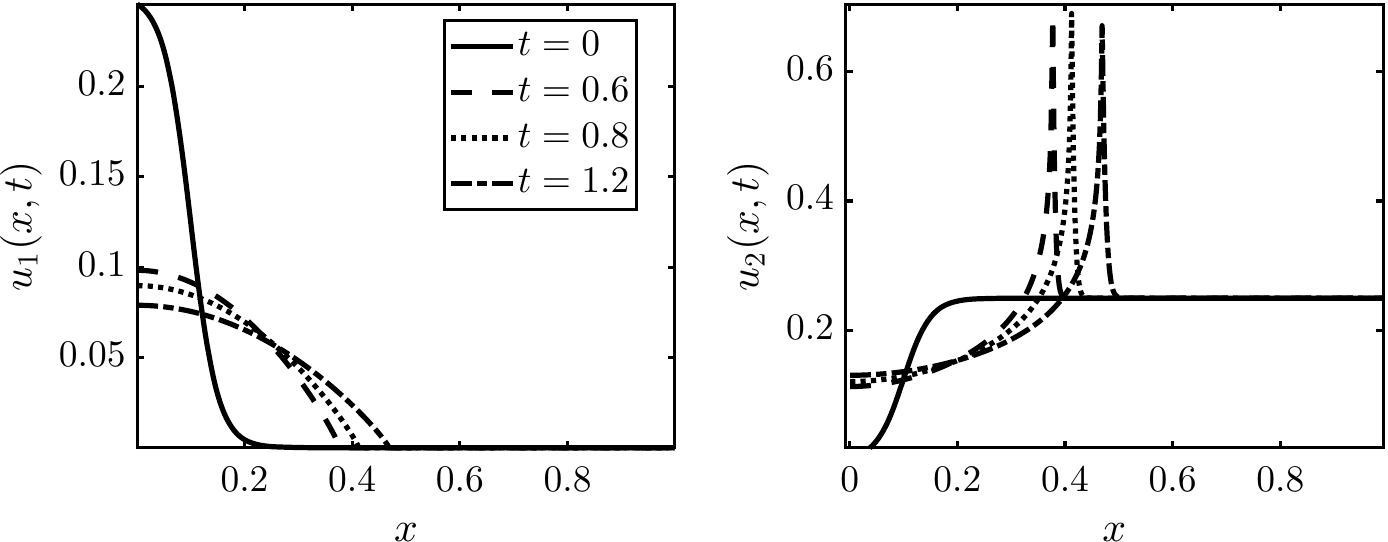} 
\caption{Volume fractions of the tumor cells (left) and the ECM (right) using $\theta=30$ (top) and $\theta=1000$ (bottom). The tumor cell front and the ECM peaks move from left to right as time increases.}
\label{fig.tgm}
\end{figure}

Figure \ref{fig.tgm2} shows the behavior of the volume fractions when the pressure parameters $\beta=\beta_c=\beta_m$ are equal. We have set $\beta=1$ and $\theta=100>\theta^*=4/\sqrt{\beta}=4$. In contrast to the previous case, the ECM fraction stays increased after the sharp rise. The relative entropy still decreases with time (Figure \ref{fig.ent}), although this can be proven only for the subcritical case $\theta<4$. We observe a linear decay and there seems to be a time interval for which the relative entropy is concave. In the subcritical case $\theta<4$, Theorem \ref{thm.decay} ensures exponential decay. 

\begin{figure}[htb]
\includegraphics[width=120mm]{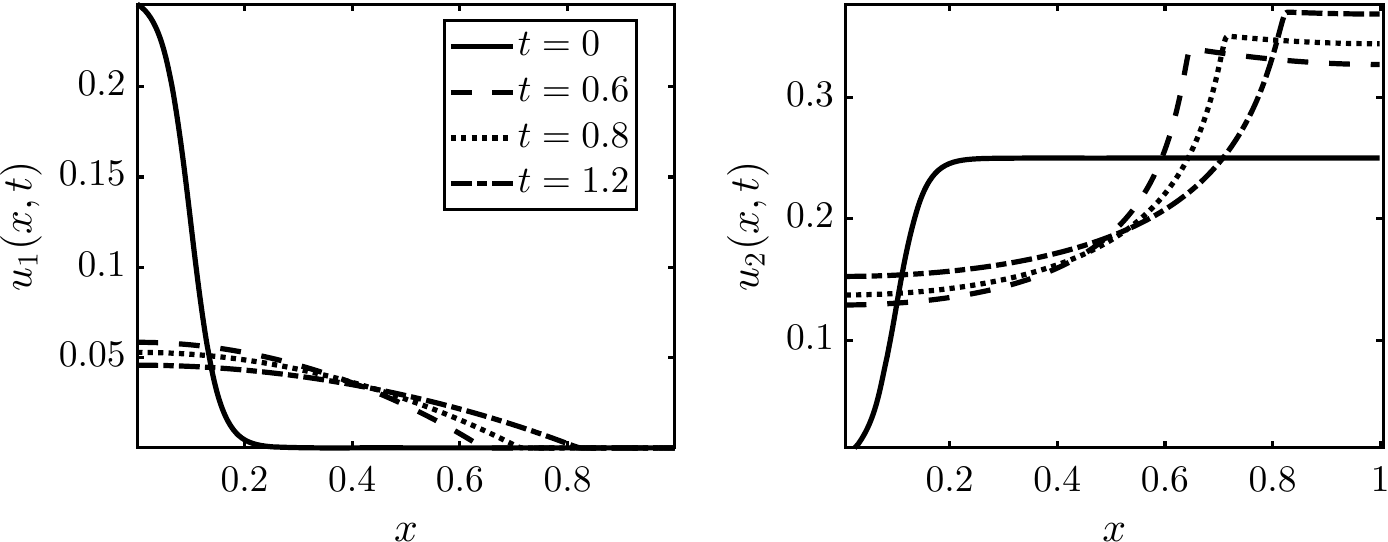}
\caption{Volume fractions of the tumor cells (left) and the ECM (right) using the symmetric values $\beta_c=\beta_m=1$ and the supercritical parameter $\theta=100$.}
\label{fig.tgm2}
\end{figure}

\begin{figure}[htb]
\includegraphics[width=120mm]{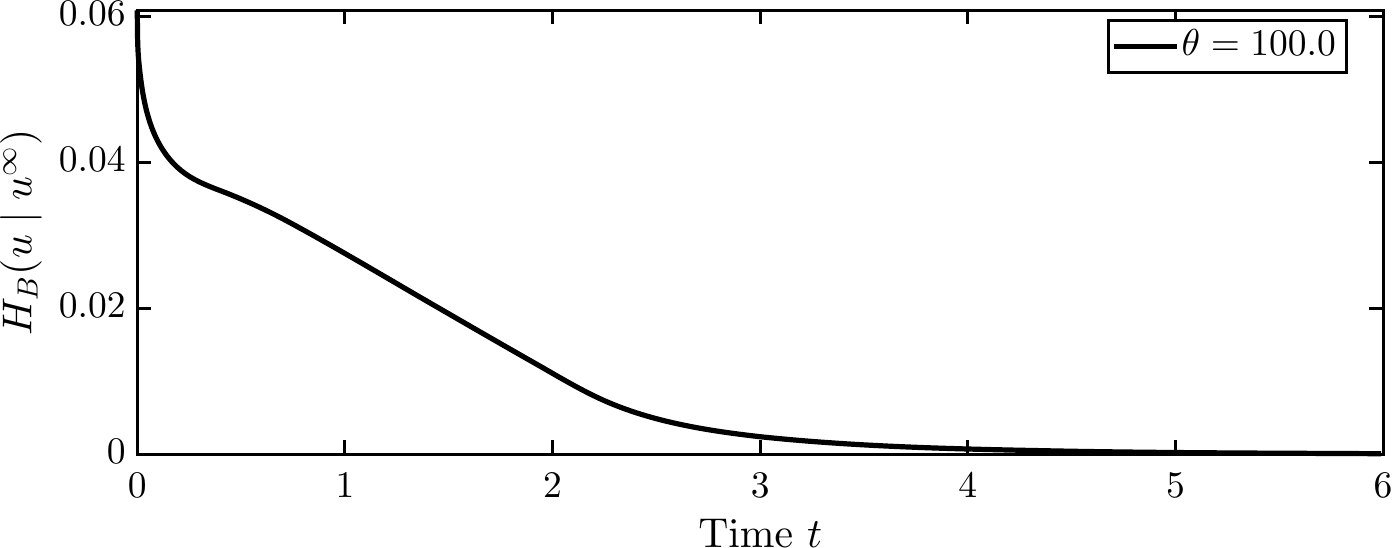}
\caption{Relative entropy for the tumor-growth model versus time using $\beta_c=\beta_m=1$ and $\theta=100$.}
\label{fig.ent}
\end{figure}

\subsection{Multiphase model}

We present some simulations for the volume-filling model \eqref{1.vfSKT} with $n=2$ and 
\begin{align*}
  q_1(u) = u_1(1+\theta_1 u_2), \quad 
  q_2(u) = u_2(1+\theta_2 u_1),
\end{align*}
where $\theta_1$, $\theta_2>0$. If $\theta_1=0$ and $\theta=\theta_2$, we recover the tumor-growth model \eqref{1.tg} with $\beta=1$. We infer the existence of weak solutions from Theorem \ref{thm.ex} if $(q_{ij})$ is positive definite or $\theta_1+\theta_2<2$. Figure \ref{fig.mp} shows the behavior of the volume fractions for $\theta_1=1$ and two values of $\theta_2$. In both cases, the ECM fraction becomes smaller near $x=0$ because of the presence of the tumor cells, and increases up to some point. When the parameter $\theta_2$ is large, we observe a behavior of the ECM fraction similar as in Figure \ref{fig.tgm2}. As in the previous example, the relative entropy decreases monotonically but not exponentially (see Figure \ref{fig.relent}). Moreover, larger values of $\theta_1$ lead to a faster decay, showing a pronounced concave region.

\begin{figure}[htb]
\includegraphics[width=120mm]{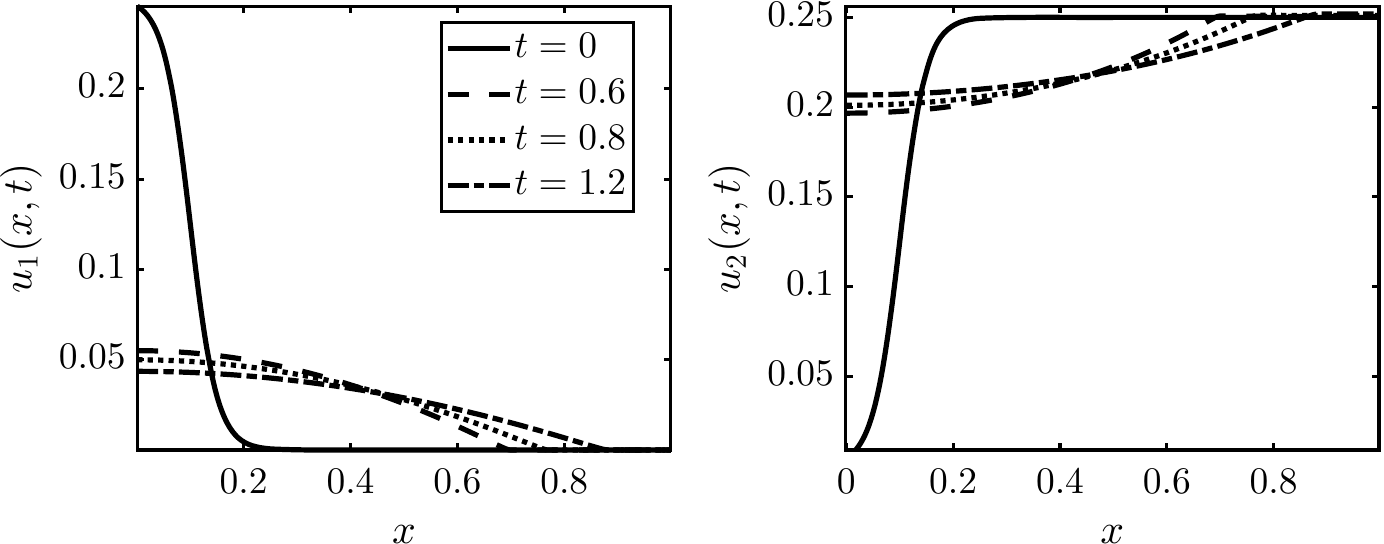}
\includegraphics[width=120mm]{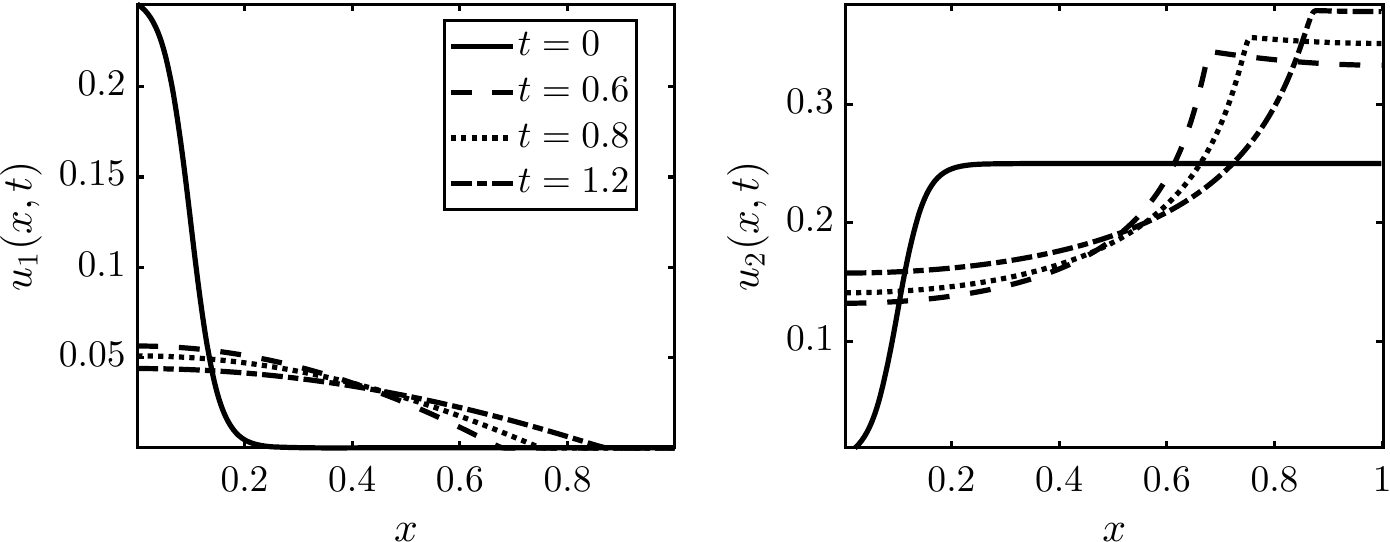}
\caption{Volume fractions of the multiphase model \eqref{1.vfSKT} using $\theta_1=1$ and $\theta_2=10$ (top) as well as $\theta_2=100$ (bottom).}
\label{fig.mp}
\end{figure}

\begin{figure}[htb]
\includegraphics[width=110mm]{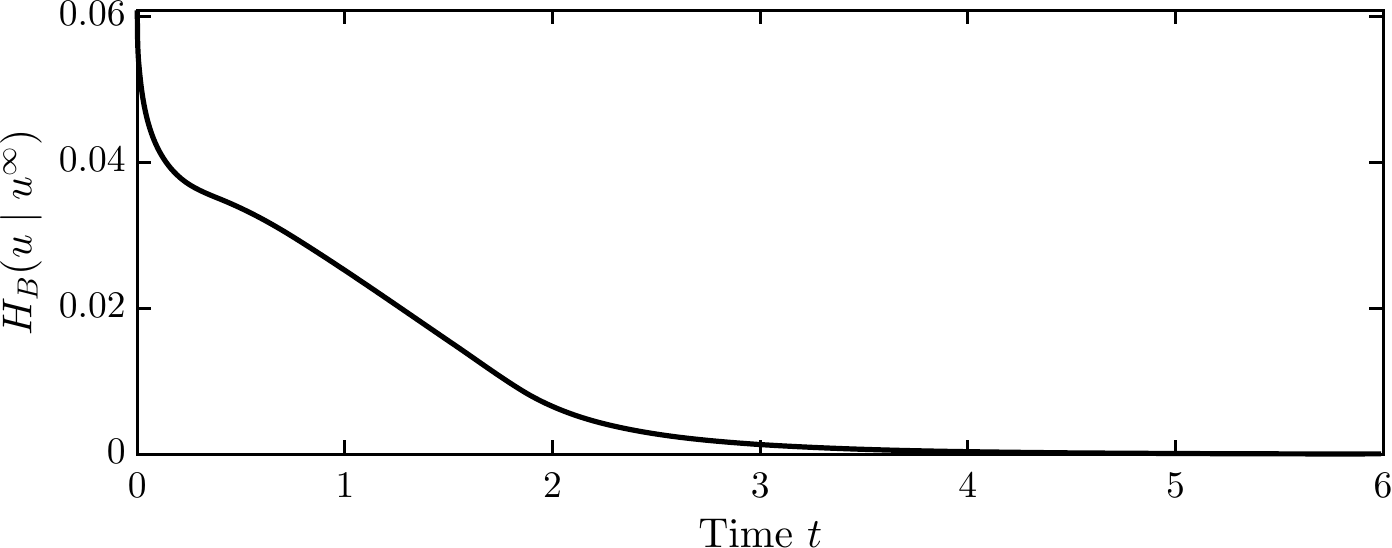}
\includegraphics[width=110mm]{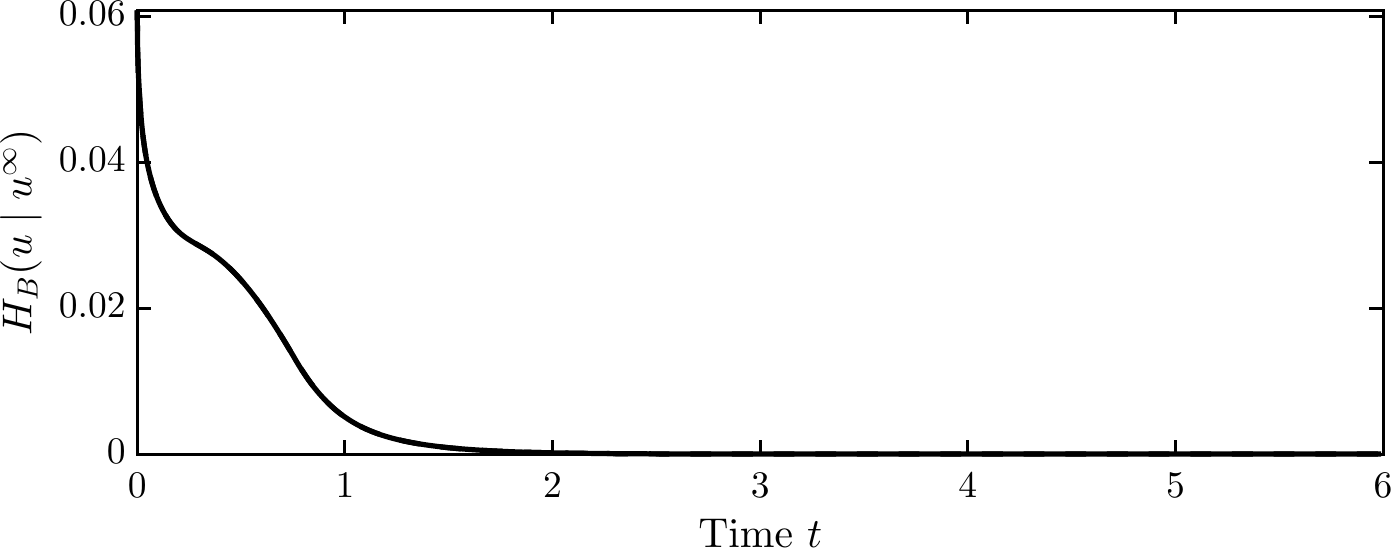}
\caption{Relative entropy for the multiphase model versus time using $\theta_1=1$, $\theta_2=100$ (top) and $\theta_1=10$, $\theta_2=100$ (bottom).}
\label{fig.relent}
\end{figure}


\begin{appendix}
\section{Proof of Proposition \ref{prop.G}}\label{app}

We compute the matrix $G(u)=h_B''(u)K(u)^{-1}A(u)$, using $h''_B(u)=(H_{ij})$ with $H_{ij}=\delta_{ij}/u_i+1/u_0$ for $i,j=1,2$, 
expression \eqref{4.K1} of $K(u)^{-1}$, and expression \eqref{4.Au} of $A(u)$. A tedious computation yields for $z\in\R^2$:
\begin{align*}
  & z^TG(u)z = a(u)z_1^2 + b(u)z_2^2 + c(u)(z_1+z_2)^2, 
  \quad\mbox{where} \\
  & a(u) = a_{11}(u)q_{11} + a_{12}(u)q_{12} + a_{21}(u)q_{21}
  + a_{22}(u)q_{22}, \\
  & b(u) = b_{11}(u)q_{11} + b_{12}(u)q_{12} + b_{21}(u)q_{21}
  + b_{22}(u)q_{22}, \\
  & c(u) = c_{11}(u)q_{11} + c_{12}(u)q_{12} + c_{21}(u)q_{21}
  + c_{22}(u)q_{22},
\end{align*} 
and the coefficients equal
\begin{align*}
  a_{11}(u) &= (k_{01} + k_{12} - 2 k_{02})u_1 + 2 k_{02}, \\
  a_{12}(u) &= k_{12}u_2 - \frac{1}{2}\big(k_{12}(u_1 + u_2)
  + k_{02}(u_0 + u_2) - k_{01} u_2\big) 
  + k_{02}u_2 \frac{u_0 + u_2}{u_1}, \\ 
  a_{21}(u) &= -\frac12\big(k_{01}(u_0+u_1) - k_{02}(u1-2u_2)
  + k_{12}(u_1-u_2)\big), \\
  a_{22}(u) &= (k_{02} - k_{12})u_2, \\
  b_{11}(u) &= (k_{01} - k_{12})u_1, \\
  b_{12}(u) &= -\frac{k_{02}}{2} + \frac{u_1}{2}(-2k_{01} 
  + k_{02} + k_{12}) + \frac{u_2}{2}(k_{01} - k_{12}), \\
  b_{21}(u) &= k_{01}u_1\frac{u_0+u_1}{u_2} - \frac12\big(
  k_{01}(u_0+u_1) - (k_{02}+k_{12})u_1 + k_{12}u_2\big), \\
  b_{22}(u) &= (k_{02}+k_{12}-2k_{01})u_2 + 2k_{01}, \\
  c_{11}(u) &= (k_{12}-k_{01})q_{11}u_1, \\
  c_{12}(u) &= \frac12\big(k_{12}(u_1+u_2) + k_{02}(u_0+u_2) 
  - k_{01}u_2\big)q_{12}, \\
  c_{21}(u) &= \frac12\big((k_{12}-k_{02})u_1 + k_{12}u_2
  + k_{01}(u_0+u_1)\big)q_{21}, \\
  c_{22}(u) &= (k_{12}-k_{02})u_2q_{22}.
\end{align*}
We determine conditions on $k_{ij}$ and $q_{ij}$ such that $a>0$, $b>0$, and $c>0$ uniformly in $u\in\dom$. Since the computations are similar for the different cases, we only present the full proof for the first case: $k_{01}<k_{12}<k_{02}$ and $k_{02}\le k_{01}+k_{12}$.

First, we show that $a(u)>0$ under the conditions stated in the proposition. The coefficient $a_{11}(u)$ is affine and decreasing in $u_1$ since $k_{01}+k_{12}-2k_{02}<0$. Hence, $a_{11}(u)\ge a_{11}|_{u_1=1} = k_{01}+k_{12}$. Next, we obtain
\begin{align*}
  a_{12} \ge u_2 k_{12} - \frac{1}{2}\big(k_{12}(u_1 + u_2)
  + k_{02}(u_0 + u_2) - k_{01} u_2\big) =: \beta(u).
\end{align*}
The function $\beta(u)$ is affine in $u$, so the minimum is attained at one of the vertices of $\dom$. A computation shows that $\beta(u)\ge\beta(0,1)=-k_{02}/2$. This shows that $a_{12}(u)\ge -k_{02}/2$. The coefficient $a_{21}(u)$ is affine. Thus, the minimum is attained at one of the vertices, giving $a_{21}(u)\ge a_{21}(0,1) = -k_{02}+k_{12}/2$. Since $k_{02}>k_{12}$, we have $a_{22}\ge 0$. Putting these conditions together and taking into account \eqref{4.case1}, we arrive at
\begin{align}\label{a}
  a(u) \ge (k_{01}+k_{12})q_{11} - \frac12 k_{02}q_{12}
  - \bigg(k_{02}-\frac{k_{12}}{2}\bigg)q_{21} > 0.
\end{align}

Second, we verify $b(u)>0$. We infer from $k_{01}<k_{12}$ and $u_1\le 1$ that $b_{11}(u)\ge -(k_{12}-k_{01})$. As $b_{12}$ is affine, the minimum is attained at one of the vertices of $\dom$, and we find that $b_{12}(u)\ge -\frac12(k_{02}+k_{12})$. We neglect the first term in $b_{21}$ and estimate the second term, leading to $b_{21}(u)\ge -\frac12(k_{01}+k_{12})$. The function $b_{22}$ is increasing in $u_2$, showing that $b_{22}(u)\ge 2k_{01}$. We have proved that
\begin{align}\label{b}
  b(u)\ge 2k_{01}q_{22} - (k_{12}-k_{01})q_{11}
  - \frac12(k_{02}+k_{12})q_{12} - \frac12(k_{01}+k_{12})q_{21} > 0.
\end{align}

Finally, we prove that $c(u)>0$. Since $k_{12}>k_{01}$, we have $c_{11}(u)\ge 0$. The minimum of the affine function $c_{12}$ equals $k_{12}q_{12}/2$, while the minimum of $c_{21}$ is $(k_{01}+k_{12}-k_{02})q_{21}/2$. Finally, using $k_{12}-k_{02}<0$, we have $c_{22}(u)\ge (k_{12}-k_{02})q_{22}$. This shows that
\begin{align}\label{c}
  c(u) \ge \frac12 k_{12}q_{12} + \frac12(k_{01}+k_{12}-k_{02})q_{21}
  - (k_{02}-k_{12})q_{22} > 0. 
\end{align}
Summarizing inequalities \eqref{a}--\eqref{c}, we conclude the positive definiteness of $G(u)$. The other cases are proved by similar arguments.
\end{appendix}



\begin{thebibliography}{11}

\bibitem{AbLe25} E.~Abdo and F.-N.~Lee. Logarithmic Sobolev inequalities for bounded domains and applications to drift-diffusion equations. {\em J. Funct. Anal.} 288 (2025), no.~110716, 13 pages.

\bibitem{AlLu83} H.-W.~Alt and S.~Luckhaus. Quasilinear elliptic–parabolic differential equations. {\em Math. Z.} 183 (1983), 311--341.

\bibitem{Ama90} H.~Amann. Dynamic theory of quasilinear parabolic equations II. Reaction--diffusion equations. {\em Differ. Int.
Eqs.} 3 (1990), 13--75.

\bibitem{BaEh18} A.~Bakhta and V.~Ehrlacher. Cross-diffusion systems with non-zero flux and moving boundary conditions. {\em ESAIM Math. Model. Numer. Anal.} 52 (2018), 1385--1415.

\bibitem{Bot11} D.~Bothe. On the Maxwell--Stefan equations to multicomponent diffusion. In: J.~Escher et al.\ (eds), {\em Progress in Nonlinear Differential Equations and their Applications}, vol.~80, pp.~81--93. Springer, Basel, 2011.

\bibitem{BuTr83} S.~Busenberg and C.~Travis. Epidemic models with spatial spread due to population migration. {\em J. Math. Biol.} 16 (1983), 181--198.

\bibitem{CCDJ25} J.~A.~Carrillo, X.~Chen, B.~Du, and A.~J\"ungel. Fluid relaxation approximation of the Busenberg--Travis cross-diffusion system. {\em Commun. Math. Phys.} 406 (2025), no.~151, 29 pages.

\bibitem{ChJu18} X.~Chen and A.~J\"ungel. A note on the uniqueness of weak solutions to a class of cross-diffusion systems. {\em J. Evol. Eqs.} 18 (2018), 805--820.

\bibitem{ChJu21} X.~Chen and A.~J\"ungel. When do cross-diffusion systems have an entropy  structure? {\em J. Differ. Eqs.} 278 (2021), 60--72.

\bibitem{GoMa98} V.~Giovangigli and M.~Massot. The local Cauchy problem for multicomponent reactive flows in full vibrational nonequilibrium. {\em Math. Meth. Appl. Sci.} 21 (1998), 1415--1439.

\bibitem{HeJu25} M.~Heitzinger and A.~J\"ungel. Weak--strong uniqueness for general cross-diffusion systems with volume filling. Submitted for publication, 2025. arXiv:2509.25978.


\bibitem{HJT22} X.~Huo, A.~J\"ungel, and A.~Tzavaras. Weak--strong uniqueness for Maxwell--Stefan systems. {\em SIAM J. Math. Anal.} 54 (2022), 3215--3252.

\bibitem{JaBy02} T.~Jackson and H.~Byrne. A mechanical model of tumor encapsulation and transcapsular spread. {\em Math. Biosci.} 180 (2002), 307--328.

\bibitem{Jue15} A.~J\"ungel. The boundedness-by-entropy method for cross-diffusion systems. {\em Nonlinearity} 28 (2015), 1963--2001.

\bibitem{Jue16} A.~J\"ungel. {\em Entropy Methods for Diffusive Partial Differential Equations}. BCAM Springer Briefs, Cham, 2016.

\bibitem{JuSt12} A.~J\"ungel and I.~Stelzer. Entropy structure of a cross-diffusion tumor-growth model. {\em Math. Models Meth. Appl. Sci.} 22 (2012), no.~1250009, 26 pages.

\bibitem{JuSt13} A.~J\"ungel and I.~Stelzer. Existence analysis of Maxwell-Stefan systems for multicomponent mixtures. {\em SIAM J. Math. Anal.} 45 (2013), 2421--2440.

\bibitem{JuZa16} A.~J\"ungel and N.~Zamponi. Qualitative behavior of solutions to cross-diffusion systems from population dynamics. {\em J. Math. Anal. Appl.} 440 (2016), 794--809.

\bibitem{LKBJS06} G.~Lemon, J.~King, H.~Byrne, O.~Jensen, and K.~Shakesheff. Mathematical modelling of engineered tissue growth using a multiphase porous flow mixture theory. {\em J. Math. Biol.} 52 (2006), 571--594.

\bibitem{KSZ21} P.~Kumar, C.~Surulescu, and A.~Zhigun. Multiphase modelling of glioma pseudopalisading under acidosis. {\em Math. Engin.} 4 (2021), 1--28.

\bibitem{LaMa23} P.~Lauren\c{c}ot and B.-V.~Matioc. Bounded weak solutions to a class of degenerate cross-diffusion systems. {\em Ann. H. Lebesgue} 6 (2023), 847--874.

\bibitem{LOK22} X.~Liu, S.~Oh, and M.~Kirschner. The uniformity and stability of cellular mass density in mammalian cell culture. {\em Front. Cell Dev. Biol.} 10 (2022), no.~1017499, 24 pages.

\bibitem{SKT79} N.~Shigesada, K.~Kawasaki, and E.~Teramoto. Spatial segregation of interacting species. {\em J. Theor. Biol.} (1979), 83--99.

\bibitem{WeKr00} J.~Wesselingh and R.~Krishna. {\em Mass Transfer in Multicomponent Mixtures}. Delft University Press, Delft, 2000.

\end{thebibliography}
\end{document}